\documentclass[12pt,a4paper,twoside,notitlepage]{article}
\usepackage[english]{babel}
\usepackage[T1]{fontenc} 
\usepackage{graphicx}
\usepackage{amsmath,amsthm,epsfig,amsfonts,bbm} 
\DeclareMathOperator{\sgn}{sgn}

\newcommand{\Frac}[2] {\frac{\mbox{\normalsize{$#1$}}}{\mbox{\normalsize{$#2$}}}}
\newcommand{\DP}[2]{\displaystyle \frac{\partial #1}{\partial #2}}
\newcommand{\pequationdeb}{$$ \left\{ \begin{minipage}[c]{130mm}}
\newcommand{\pequationfin}{\end{minipage}
                           \right. $$}
\newcommand{\vite}{\mathbf{u}}
\newcommand{\srb}{\sqrt{R\sqrt{b}}}
\newcommand{\Rb}{R\sqrt{b}}
\newcommand{\beq}     {\begin{equation}}
\newcommand{\enq}     {\end{equation}}
\newcommand{\be}    {\begin{enumerate}}
\newcommand{\ee}    {\end{enumerate}}

\newcommand{\Bb}

\newcommand{\RE}{\mbox{Re}}
\newcommand{\sqb}{\sqrt{\beta}}

\newcommand{\Int}     {\displaystyle \int}
\newcommand{\Sup}     {\displaystyle \sup}

\newtheorem{theorem}{Theorem}%[section]
\newtheorem{lemma}[theorem]{Lemma}%[section]
\newtheorem{remark}[theorem]{{\em Remark}}%[section]
\newtheorem{proposition}[theorem]{Proposition}%[section]

\begin{document}
\thispagestyle{empty}
\setcounter{page}{1}

\title{Derivation of a viscous Boussinesq system for surface water waves}

\author{Herv\'e V.J. Le Meur$^{1,2}$\\
$^1$ CNRS, Laboratoire de Math\'ematiques d'Orsay, Orsay cedex, F-91405\\
$^2$ univ Paris-Sud, Orsay cedex, F-91405. \\
\texttt{Herve.LeMeur@math.u-psud.fr}}

\date{October 30, 2014}

%\address{CNRS, Laboratoire de Math\'ematiques d'Orsay, Orsay cedex, F-91405; univ Paris-Sud, Orsay cedex, F-91405.}
%\email{Herve.LeMeur@math.u-psud.fr}

\maketitle

\begin{abstract}
  In this article, we derive a viscous Boussinesq system for surface
  water waves from Navier-Stokes equations for non-vanishing initial
  conditions. We use neither the irrotationality assumption, nor the
  Zakharov-Craig-Sulem formulation. During the derivation, we find the
  bottom shear stress and also the decay rate for shallow water. In order to justify our derivation, we derive the
  viscous Korteweg-de Vries equation from our viscous Boussinesq
  system and compare it to the ones found in the bibliography. We also
  extend the system to the 3-D geometry.
\end{abstract}

%\subjclass{76N20, 74J15, 76M45}
%\keywords{water waves, shallow water, Boussinesq system, viscosity, KdV equation}

\underline{Subject Class:} 35Q35, 76B15, 76N20, 76M45, 35Q53

\section{Introduction}

\subsection{Motivation}

The propagation of water waves over a fluid is a long run issue
in mathematics, fluid mechanics, hydrogeology, coastal engineering,
... In the case of an inviscid fluid, the topic stemmed many researches and
even broadened with time. Various equations have been proposed to
model this propagation of water waves. Since the full problem is very
complex, the goal is to find reduced models on simplified
domains with as little fields as possible, should they be valid only
in an asymptotic regime.

This article is a step forward in the direction of a rigorous
derivation of an asymptotic system for surface water waves in the
so-called Boussinesq regime, taking into account the viscosity. While
viscous effects can be neglected for most oceanic situations, they
cannot be excluded for surface waves in relatively shallow channels.

In the inviscid potential case, the complete and rigorous
justification of most asymptotic models for water waves has been
thoroughly carried out and summarized in the book \cite{Lannes_2013}
and the bibliography therein. This book includes the proof of the
consistency and stability of some models, the proof of the existence
of solutions both of the water waves systems and of the asymptotic
models on the relevant time scales and the proof of ``optimal'' error
estimates between these two solutions. The curlfree assumption allows
to use the Zakharov-Craig-Sulem formulation of the water waves system
and facilitates the rigorous derivation of the models, through
expansions of the Dirichlet to Neumann operator with respect to a
suitable small parameter.

Things are more delicate when viscosity is taken into account and a
complete justification of the asymptotic models is still lacking. The
main difficulty, for not only a derivation but also for a rigorous
proof, arises from the matching between the boundary layer solution
coming from the bottom and the "Euler" solution in the upper part of
the flow.

In this article, we derive an asymptotic system (Boussinesq system)
for the viscous flow in a flat channel of water waves in the
Boussinesq regime, that is to say in the long wave, small amplitude
regime with an {\em ad hoc} balance between the two effects.

\subsection{Literature}

When deriving models of water waves in a channel and taking viscosity
into account, numerous pitfalls must be avoided in order to be rigorous.

Since there are various dimensionless parameters, a linear study must
be done so as to determine the most interesting regime between the
parameters. One must also either assume linearized Navier-Stokes
Equation (NSE), or justify that the nonlinear terms can be
dropped. This is not so obvious because numerous authors extend the
inviscid theory by assuming the velocity to be the sum of an inviscid
velocity and a viscous one. Then they force only one condition (for
instance the vanishing velocity on the bottom) to be satisfied by the
total velocity, once the inviscid velocity is assumed unchanged by
viscosity. This assumption deserves to be jutified.

At a certain level of the derivation, a heat-like equation
arises. Most people solve it with a time Fourier transform while the only
physical problem is a Cauchy one, so with an initial condition. The
only possibility is to use either Laplace (in time) transform or a
sine-transform (in the vertical dimension) with a complete treatment
of the initial condition.

One must also derive the bottom shear stress because it is meaningful
for the physicists who deal with sediment transport.

Last, the order up to which the expansion is done must be consistent
throughout the article.

To the best of our knowledge, no article does all this. Yet various
articles have been written on this topic. Let us review those that
retained our attention and interest.\\

Boussinesq did take viscosity into account in 1895
\cite{Boussinesq_1895_1}. Lamb \cite{Lamb_1906} also derived the decay
rate of the linear wave amplitude by a dissipation calculation (done
in paragraph 348 of the sixth edition of \cite{Lamb_1906}) and by a
direct calculation based on the linearized NSE (paragraph 349 of the
sixth edition of \cite{Lamb_1906}). Both of them used linearized NSE
on deep-water and computed the dispersion relation. We do not know who
is the first. The imaginary part of the phase velocity gave the decay
rate:
\begin{equation}
\label{intro1}
\DP{A}{t}= -2\nu k^2 A,
\end{equation}
where $A$ is the amplitude of the wave, $\nu$ the kinematic viscosity
and $k$ the wavenumber.

%{\bf Longuet-Higgins
%  writes \cite{LH_53} to improve Stoke's study of the mass transport
%  velocity (steady second order drift velocity apart from the orbital
%  motion) of a fluid. In his article, he uses viscosity to study the
%  boundary layers at the bottom and free surface, so as to guess the
%  real motion in the interior of a channel and the dynamics of Stoke's
%  drift. Inclusion of viscosity enables him to match observations.
%
%  Longuet-Higgins writes two articles \cite{LH_92_a} and
%  \cite{LH_92_b} fourty years later to study the weakly damped Stokes
%  waves, and the way viscosity generates vorticity in the boundary
%  layer of the free surface at the forward face of the crest of
%  gravity waves. This vorticity is generated by the constraint of
%  vanishing shear stress at the free boundary, mainly where curvature
%  is large.
%
%In \cite{LH_92_b}, Longuet-Higgins determines an (additionnal
%and theoretical) surface elevation generated by the viscous tangential
%stress. This elevation produces an additionnal normal stress that
%justifies the relation (\ref{intro1}).
%
%But all these studies (Longuet-Higgins and Ruvinsky {\em et al.}) deal
%with Stokes waves which will not interest us furthermore.}

In \cite{Ott_Sudan_70}, Ott and Sudan made a formal derivation (in
nine lines) of a dissipative KdV equation (different from ours). They
used the linear damping of shallow water waves already given by
Landau-Lifschitz. This led them to an additional term to KdV, which
looks like a half integral. They also found once again the damping in
time of a solitary wave over a finite depth as $(1+T)^{-4}$ (already
found by \cite{Keulegan_48}, and later by \cite{KM_75}, \cite{Mei_83},
\cite{Johnson_97} (p. 374)).

J. Byatt-Smith studied the effect of a laminar viscosity (in the
boundary layer where a laminar flow takes place) on the solution of an
undular bore \cite{Byatt-Smith_71}. He found the (almost exact)
Boussinesq system of evolution with a half derivative but with no
treatment of the initial condition. He did an error when providing the
solution to the heat equation: his convolution in time is over
$(0,+\infty)$ instead of $\mathbb{R}$ (Fourier convolution) or $(0,t)$
(Laplace convolution).

In 1975, Kakutani and Matsuuchi \cite{KM_75} found a minor error in
the computation of \cite{Ott_Sudan_70}. They started from the NSE and
performed a clean boundary layer analysis. First, they made a linear
analysis that gave the dispersion relation and, under some assumptions,
the phase velocity as a function of both the wavenumber of the wave
and the Reynold's number Re. They distinguished various regimes of Re
as a function of the classical small parameter of the Boussinesq
regime. Then, they derived the corresponding viscous KdV equation. We
want to stress that, at the level of the heat equation, they used a
time Fourier transform. As a consequence, they may not use any
initial condition. So, the problem they solve is not the Cauchy's one.

In \cite{Matsuuchi_76}, one of the authors of the previous article
\cite{KM_75} tried to validate the equation they had been led to. He
showed that their ``modified K-dV equation agrees with
Zabushy-Galvin's experiment with respect to the damping of solitary
waves, while it produces disagreement in their phases'' (see the
conclusions). One might object that the numerical treatment seems
light because the space step was between one and 10 percent, the
numerical relaxation was not very efficient, the unbounded domain was
replaced by a periodic one though there is ``non-locality of the
viscous effect'' (p. 685), there was no numerical validation of the
full algorithm, and the regime was not the Boussinesq one (the
dispersion's coefficient was about 0.002 and the viscous coefficient
was 0.03). Moreover, the phase shift numerically measured was given
with three digits while the space step was of the order of magnitude
of some percents. The author, very fairly, added that ``the phase shift
obtained by the calculations is not confirmed by [the] experiments''.

In 1987, Khabakhpashev \cite{Khabakhpashev_87} extended the derivation
of the viscous KdV evolution equation to the derivation of a viscous
Boussinesq system. He studied the dispersion relation and predicted a
reverse flow in the bottom, in case of the propagation of a soliton
wave. He used a Laplace transform (instead of Fourier as \cite{KM_75}
did) with vanishing initial conditions since he assumed starting from
rest. Although he stressed this assumption, he acknowledged that ``the
time required for the boundary layer to develop over the entire
thickness of the fluid [is] much greater than the characteristic time
of the wave process''. The equations were not made dimensionless, so
the right regime was not discussed and a very inefficient numerical
method was used (Taylor series expansion is replaced in the
convolution term).

In the book \cite{Johnson_97} (part 5 pp. 356--391), Johnson found the
same dispersion relation as \cite{KM_75}, studied the attenuation of
the solitary wave by a multiscale derivation, reached a heat
equation, but solved it only with vanishing initial condition. He
exhibited a convolution with a square root integrated on $(0,+\infty)$
(like Byatt-Smith \cite{Byatt-Smith_71}). Some numerical simulations
(already partialy done by \cite{Byatt-Smith_71}) enabled him to
recover the mecanism of undular bore slightly damped.

$ $ \\

%{\bf Mei and Liu [JFM73] studied thoroughly the boundary layer of
%a viscous fluid at the side-walls (and bottom), at the free surface,
%and also at the free surface meniscus. They assumed a linearized
%NSE and checked the balance of the energy fluxes
%between these three boundaries and the interior. Taking the meniscus
%into account enabled them to explain the discrepancies of previous
%calculations.}

Later, Liu and Orfila wrote a seminal article \cite{Liu_Orfila_04}
(LO hereafter) in which they studied water waves in an infinite channel
(so without meniscus). They derived a Boussinesq system with an additionnal
half integration (seen as a convolution), and an initial condition
assumed to be vanishing, but implicitely added to the system when
numerical simulation must be done.

More precisely, the authors took a linearized Navier-Stokes fluid,
used the Helmholtz-Leray decomposition and defined the parameters
(index $LO$ denotes their parameters):

\[
\begin{array}{l}
\alpha^2_{LO}= \nu / \left( l\sqrt{g h_0} \right),\\
\varepsilon_{LO}=A/ h_0,\\
\mu_{LO}=h_0 / l,
\end{array}
\]
where the following notations will be used throughout the present
article: $A$ is the characteristic amplitude of the wave, $h_0$ is the
mean height of the channel, $g$ is the gravitational acceleration, and $l$
is the characteristic wavelength of the wave. They made
expansions up to order $\alpha_{LO}$, which square is a kind of a
Reynold's number inverse. They used the classical Boussinesq
approximation: $\varepsilon_{LO} \simeq \mu_{LO}^2$, but they also set
the link between the viscosity and $\varepsilon_{LO}$ by requiring
$O(\alpha_{LO}) \simeq O(\varepsilon_{LO}^2) \simeq O(\mu_{LO}^4)$
without further justification. Although ``the boundary layer
thickness is of $O(\alpha_{LO})$'', they stretched the coordinates by a
larger factor $\alpha_{LO} / \mu_{LO}\simeq \mu_{LO}^3$ (see their
(2.9)). More important, and maybe linked, they kept the $\alpha_{LO}
\mu_{LO}$ terms (in their (2.8) or (2.21) for instance) and yet dropped
$o(\alpha_{LO})$ terms !  This can explain why their final system
(3.10-3.11) had a $\alpha_{LO}/\mu_{LO}=O(\varepsilon_{LO}^{3/2})$
term before the half integration, while we will justify an
$O(\varepsilon_{LO})$ term for our system.

Let us stress that assuming $\alpha_{LO}^2=\varepsilon_{LO}^4$ as did
\cite{Liu_Orfila_04} amounts to Re $=\varepsilon_{LO}^{-7/2}$ with our
(further redefined) Reynold's number: Re$=\nu/(h_0\sqrt{g h_0})$,
while we prove below that the regime at which gravity and viscosity
are both relevant is Re $=\varepsilon_{LO}^{-5/2}$. Our regime was
also exhibited by \cite{KM_75}, \cite{Byatt-Smith_71},
\cite{Johnson_97}. So, \cite{Liu_Orfila_04} studied a regime
different from ours.

Last, they claimed the shear stress at the bottom to be:
\[
\tau_{bottom} =-\frac{1}{2\sqrt{\pi}}\int_0^t \frac{u(x,T)}{\sqrt{(t-T)^3}} {\rm d } T,
\]
where $u(x,T)$ is the depth averaged horizontal velocity. Indeed this
integral is infinite as they acknowledged in a later corrigendum where they
claimed the right formula to be:
\[
\begin{array}{rl}
\tau_{bottom}  = \frac{1}{\sqrt{\pi}} \frac{u(x,0)}{\sqrt{t}}+\frac{1}{\sqrt{\pi}}\int_0^t\frac{u_{,T}}{\sqrt{t-T}} {\rm d} T,
\end{array}
\]
where $u_{,T}$ denotes the time derivative. However, they did not
provide any justification. Moreover, their solution (2.15) to the heat
equation, computed in \cite{Mei_95} (pp. 153--159), assumed vanishing
initial condition. So the treatment of the initial condition was not
done. One of our goals in the present article is precisely to provide
a better treatment of this initial condition. In this article
\cite{Liu_Orfila_04}, the authors also raised the question of the eligible
boundary condition. Indeed, they considered to be well-known that for a
laminar boundary layer, the phase shift between the bottom shear
stress and the free stream velocities is $\pi/4$. So it prohibits any
bottom condition of the Navier type $\tau_{xy} =-k u_{bottom}$ as is
sometimes assumed (and not derived).

Although we presented some criticisms, we acknowledge the modeling,
derivation and explanations of this article are insightful and, last but
not least, very well written. Yet our criticisms apply to all subsequent
articles of the same vein.

In \cite{Liu_Park_Cowen_07}, Liu {\em et al.} experimentaly validated
LO's equations in the particular case of a solitary wave over a
boundary layer. By Particle Image Velocimetry (PIV), they measured the
horizontal velocity in the boundary layer over which the solitary wave
run and confirmed the theory.

In \cite{Liu_Simarro_Vandever_Orfila_06}, Liu {\em et al.} extended the
derivation of the viscous Boussinesq system of \cite{Liu_Orfila_04} to
the case of an unflat bottom. They compared the viscous damping and
shoaling of a solitary wave, propagating in a wave tank from the
experimental and numerical point of view. They provided a condition on
the slope of the bottom and paid attention to the meniscus on
the sidewalls of the rectangular cross section.

In \cite{Liu_Chan_07}, Liu and Chan used the same process to study the
flow of an inviscid fluid over a mud bed modeled by a {\em very}
viscous fluid. They also studied the damping rate of progressive
linear waves and solitary waves. In \cite{Park_Liu_Clark_08}, Park
{\em et al.}  validated this model with experiments. They also studied
the influence of the ratio of the ``mud bed thickness and the
wave-induced boundary-layer thickness in the mud bed''.\\

%In a very separate way, Wang and Joseph \cite{Wang_Joseph_06} found
%back the Boussinesq-Lamb decay rate of free gravity waves of a viscous
%fluid over an infinite depth. They took linearized NSE
%and used the Leray-Helmholtz decomposition. They determined
%a (viscous) pressure correction so as to balance the normal stress. Oddly,
%their viscous velocity is curl-free. Since they considered only a new
%pressure, they could not satisfy the full NSE. Such a
%modeling is mainly motivated by satisfying some equations, yet, it
%gives good results since the authors can reproduce the decay rate of
%Boussinesq-Lamb over an infinite depth flow.

In 2008, Dias {\em et al.} \cite{Dias_Dyachenko_Zakharov_08} took the
(linearized) NSE of a deep water flow with a free boundary and used
the Leray-Helmholtz decomposition. Both Bernoulli's equation, through
an irrotationnal pressure, and the kinematic boundary condition were
modified. Then they made an {\em ad hoc} modeling for the nonlinear
term. Starting from such a model, they provided the evolution equation
for the enveloppe $A$ of a Stokes wavetrain which, in case of an
inviscid fluid, is the Non-Linear Schr\"odinger (NLS) equation. The
provided equation happens to be a commonly used dissipative
generalization of NLS.

Although it was published earlier (2007), \cite{Dutykh_Dias_CRAS} was a
further development of \cite{Dias_Dyachenko_Zakharov_08} to a
finite-depth flow. In this article, the authors still linearized NSE
and generalized by including additional nonlinear terms.

In a later article \cite{Dutykh_09_a}, D. Dutykh linearized NSE and
worked on dimensionned equations, considering the viscosity $\nu$ to
be small (in absolute value). The author generalized by ``including
nonlinear terms'' and reached a viscous Boussinesq system (his
(11-12)). Making this system dimensionless triggered very odd terms
and its zeroth order was no longer the wave equation. He further
derived a KdV equation by making a change of variable in space (but
not the associated change in time $\tau=\varepsilon t$). He also made
a study of the dispersion relation by assuming an exponential function
ansatz of the type $e^{i(kx-\omega t)}$, but then he froze the half
derivative term. Indeed it is well-known that such an ansatz amounts
to make a Fourier or Laplace transform. Here, the Fourier/Laplace
transform of the half derivative of $u$ is very simple: $\mid \xi
\mid^{1/2} \hat{u}$ and could have been used instead of freezing
this half derivative term.\\

In \cite{Chen_et_al_10}, Chen {\em et al.} investigated the
well-posedness and decay rate of solutions to a viscous KdV equation
which has a nonlocal term that is the same as Liu and Orfila's
\cite{Liu_Orfila_04} and \cite {Dutykh_09_a}, but not the same as
\cite{KM_75} nor the same as ours. The theoretical proofs were made
with no dispersive term ($u_{xxx}$), but with a local dissipative term
($u_{xx}$). The tools were either theoretical by finding the kernel
and studying its decay rate, or numerical. In the numerical study, they
took the dispersive term into account. As expected, they noticed that
the ``local dissipative term produces a bigger decay rate when
compared with the nonlocal dissipative term''.

In \cite{Chen_Dumont_Goubet_12}, the authors proved the global
existence of solution to the viscous KdV derived by \cite{KM_75} with
the dispersive term and even, for sufficiently small initial
conditions, without this dispersive term. In addition, they
numerically investigated the decay rate for various norms.\\

In the present article, we first make a linear study of NSE in our
domain (Section 2). We compute the dispersion relation and state
various asymptotics that give different phase velocities, and so we
give the decay rate in finite depth. In Section 3, we make the formal
derivation of the viscous Boussinesq system by splitting the upper
domain and the bottom one (the boundary layer). The explicit shear
stress at the bottom is computed. In intermediate computations, there
remains evaluations of the velocity at various heights which are
expanded so as to replace these terms by the velocity at a generic
height $z$ without the assumption of irrotationality. This enables to
state the viscous Boussinesq system. In Section 4, we state the 2-D
system, and cross-check with \cite{KM_75} that we get a similar
viscous KdV equation. We also discuss the various viscous KdV
equations proposed in the bibliography.

\section{Linear theory}

In order to make a linear theory, we need first to obtain
dimensionless equations. This is done in the next subsection. Then we
investigate two asymptotics in the following subsections.

\subsection{Dimensionless equations}

Let us denote $\tilde{\vite}=(\tilde{u},\tilde{w})$ the velocity of a
fluid in a 2-D domain $\tilde{\Omega}= \{ (\tilde{x},\tilde{z}) \; /
\; \tilde{x} \in \mathbb{R}, \; \tilde{z} \in
(-h,\tilde{\eta}(\tilde{x},\tilde{t}))\}$. So we assume the bottom is
flat and the free surface is characterized by $\tilde{z} =
\tilde{\eta}(\tilde{x},\tilde{t})$ with
$\tilde{\eta}(\tilde{x},\tilde{t})> -h$ (the bottom does not get dry).
Let $\tilde{p}$ denote the pressure and
$\tilde{\mathbf{D}}[\tilde{\vite}]$ the symmetric part of the velocity
gradient. The dimensionless domain is drawn in Fig. \ref{fig1}.
\begin{figure}[htbp]
\begin{center}
\includegraphics[width=7cm]{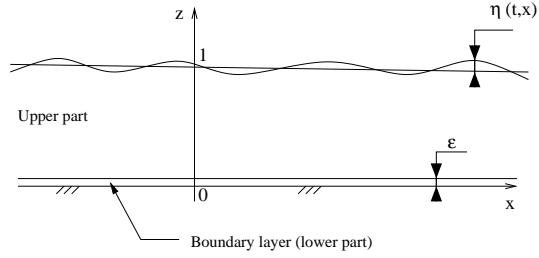}
\end{center}
\caption{The dimensionless domain}
\label{fig1}
\end{figure}
We also denote $\rho$ the density of the fluid, $\nu$ the viscosity of
the fluid, $g$ the gravitational acceleration, $\mathbf{k}$ the unit
vertical vector, $\mathbf{n}$ the outward unit normal to the upper
frontier of $\tilde{\Omega}$, $\tilde{p}_{atm}$ the atmospheric
pressure. The original system reads:
\begin{equation}
\label{Ga.1}
\left\{
\begin{array}{ll}
\rho \left( \DP{\tilde{\vite}}{\tilde{t}} + \tilde{\vite} . \tilde{\nabla} \tilde{\vite} \right) -\nu \tilde{\Delta} \tilde{\vite} +\tilde{\mathbf{\nabla}} \tilde{p} = -\rho g \mathbf{k} & \mbox{ in } \tilde{\Omega}, \\
\widetilde{\mbox{div }} \tilde{\vite} = 0 & \mbox{ in } \tilde{\Omega}, \\
\left( -\tilde{p} \mathbf{I} +2 \nu \mathbf{\tilde{D}}[\tilde{\vite}] \right) . \mathbf{n}=-\tilde{p}_{atm} \mathbf{n} & \mbox{ on } \tilde{z} = \tilde{\eta}(\tilde{x},\tilde{t}), \\
\tilde{\eta}_{\tilde{t}} +\tilde{u} \tilde{\eta}_{\tilde{x}} -\tilde{w} = 0 & \mbox{ on } \tilde{z} = \tilde{\eta}(\tilde{x},\tilde{t}), \\
\tilde{\vite} = 0 & \mbox{ on } \tilde{z} = -h,
\end{array}
\right.
\end{equation}
where we write the second order tensors and the vectors with bold
letters. Of course, we need to add an initial condition and conditions
at infinity.

So as to get dimensionless fields and variables, we need to choose a
characteristic horizontal length $l$ which is the wavelength (roughly
the inverse of the wave vector), a characteristic vertical length $h$
which is the water's height, and the amplitude $A$ of the propagating
perturbation. Moreover, we denote $U, W, P$ the characteristic
horizontal velocity, vertical velocity and pressure respectively. We
may then define:
\begin{equation}
\label{Ga.2}
\begin{array}{c}
c_0 = \sqrt{gh}, \; \alpha = \Frac{A}{h}, \; \beta =\Frac{h^2}{l^2}, \; U =\alpha c_0, \\
W=\Frac{Ul}{h}=\Frac{c_0 \alpha}{\sqb}, \; P=\rho g A, \; \mbox{Re} =\Frac{\rho c_0 h}{\nu},
\end{array}
\end{equation}
where $c_0$ is the phase velocity. As a consequence, one may make
the fields dimensionless and unscaled:
\begin{equation}
\label{Ga.3}
\tilde{u} = Uu, \; \tilde{w}= W w,\; \tilde{p} = \tilde{p}_{atm} -\rho g \tilde{z} + P p, \; \tilde{\eta} = A \eta,
\end{equation}
and the variables:
\begin{equation}
\label{Ga.4}
\tilde{x} = l x, \; \tilde{z}= h(z-1), \; \tilde{t}= t \, l/c_0.
\end{equation}
With these definitions, the new system with the new fields and
variables writes in the new domain $\Omega_t=\{(x,z), x \in
\mathbb{R}, \; z\in (0, 1+\alpha \eta(x,t)) \}$ and the new outward
unit normal still denoted $\mathbf{n}$:
\begin{equation}
\label{Ga.5}
\left\{
\begin{array}{ll}
u_t+\alpha u u_x + \Frac{\alpha}{\beta} w u_z-\Frac{\sqrt{\beta}}{\mbox{Re}} u_{xx} -\Frac{1}{\mbox{Re }\sqrt{\beta}}u_{zz} +p_x =0 & \mbox{ in } \Omega_t, \\
w_t +\alpha u w_x + \Frac{\alpha}{\beta}w w_z -\Frac{\sqrt{\beta}}{\mbox{Re}}w_{xx} -\Frac{1}{\mbox{Re }\sqrt{\beta}}w_{zz} +p_z =0 & \mbox{ in } \Omega_t,  \\
\beta u_x +w_z =0& \mbox{ in } \Omega_t, \\
(\eta-p) \mathbf{n} + \Frac{1}{\mbox{Re}} \left( \begin{array}{cc}
2 u_x \sqrt{\beta} & u_z+w_x \\
u_z+w_x & 2/ \sqrt{\beta} w_z \end{array} \right) . \mathbf{n} = 0 & \mbox{ on } z=1+\alpha \eta, \\
\eta_t+\alpha u \eta_x-\frac{1}{\beta} w = 0 & \mbox{ on } z=1+\alpha \eta, \\
\vite=0 & \mbox{ on } z=0.
\end{array}
\right.
\end{equation}
Like Kakutani and Matsuuchi \cite{KM_75}, we could have eliminated
$\eta -p$ in one of the two equations of stress continuity at the free
boundary. After simplification by 1/Re, this would have led
us to the ``simplified'' system:
\[
\left\{
\begin{array}{l}
\eta -p +(-\alpha \eta_x (u_z+w_x)-2u_x \sqrt{\beta})/\mbox{Re}=0, \\
(1-(\alpha \eta_x)^2)(u_z+w_x) = 4\alpha \sqrt{\beta}\eta_x u_x.
\end{array}
\right.
\]

Notice that the number of dynamic conditions is linked to the
Laplacian's presence. If, in a subdomain, the flow is inviscid (Euler
or Re $\rightarrow \infty$), then one must not keep the two above
equations. Yet, once we have simplified the 1/Re term above we
might forget that the second equation must be taken off as if there
remained a $1/$Re term before every term. So this ``simplification''
can be misleading.

Unlike us, the authors of \cite{KM_75} use the same characteristic
length in the two space directions and so, for them, $h/l=1$. Our
vertical velocity (scaled by $W$) is not the same as in
\cite{KM_75}. Our choice of scale for $W$ raises some $\sqb$ terms
that \cite{KM_75} avoids. It suffices to set $\beta = 1$ in our
equations to get those of \cite{KM_75}. Although the authors make
their system dimensionless, they did not really unscale the fields nor
the variables. Our fields are unscaled and so are of the order of
unity.

Our characteristic pressure is $P=\rho g A$ while \cite{KM_75} use
$\rho g h$. This explains why \cite{KM_75} has an $\alpha$ more
before the pressure $p$ in their equations.

\subsection{The dispersion relation}
We are looking for small fields. So we linearize the system
(\ref{Ga.5}) and get:
\begin{equation}
\label{Ga.6}
\left\{
\begin{array}{ll}
u_t-\Frac{\sqrt{\beta}}{\mbox{Re}} u_{xx} -\Frac{1}{\mbox{Re }\sqrt{\beta}}u_{zz} +p_x =0 & \mbox{ in } \mathbb{R} \times [0,1], \\
w_t -\Frac{\sqrt{\beta}}{\mbox{Re}}w_{xx} -\Frac{1}{\mbox{Re }\sqrt{\beta}}w_{zz} +p_z =0 & \mbox{ in } \mathbb{R} \times [0,1], \\
\beta u_x +w_z =0& \mbox{ in } \mathbb{R} \times [0,1], \\
\eta -p -\Frac{2 \sqrt{\beta} u_x}{\mbox{Re}}=0 & \mbox{ on } z=1, \\
u_z+w_x = 0 & \mbox{ on } z=1, \\
\eta_t-\frac{1}{\beta} w = 0 & \mbox{ on } z=1,\\
\vite=0 & \mbox{ on } z=0.
\end{array}
\right.
\end{equation}
First, we eliminate the pressure from (\ref{Ga.6})$_1$ and (\ref{Ga.6})$_2$:
\[
u_{zt}-\frac{\sqb}{\RE}u_{xxz}-\frac{1}{\RE \sqb} u_{zzz} -w_{xt}+\frac{\sqb}{\RE} w_{xxx}+\frac{1}{\RE \sqb} w_{xzz}=0.
\]
Then we eliminate $u$ from the previous equation thanks to
(\ref{Ga.6})$_3$ by differentiating with respect to $x$ and after some
simplifications:
\begin{equation}
\label{Ga.7}
(\partial_z^2 + \beta \partial_x^2)(\partial_z^2+\beta \partial_x^2 -\RE \sqb \partial_t)w=0.
\end{equation}
If $w$ is of the form $\mathcal{A}(z) \exp{ik(x-ct)}$ with a
non-negative $k$ and a (complex) phase velocity $c$, we can define
a parameter with non-negative real
part similar to the one used by \cite{KM_75}:
\begin{equation}
\label{Ga.7.5}
\mu^2=\beta k^2 - \RE \sqb ikc.
\end{equation}
Thanks to this notation, the solutions of (\ref{Ga.7}) are such that
\begin{multline}
\label{Ga.8}
\mathcal{A}(z)=  C_1 \cosh{\sqb k(z-1)} +C_2 \sinh{\sqb k(z-1)} \\
    +C_3 \cosh{\mu (z-1)} +C_4 \sinh{\mu(z-1)}.
\end{multline}

Up to now we have eliminated $u$ and $p$ only in the volumic
equations. We still have to use the boundary conditions of
(\ref{Ga.6}) to find the conditions on the remaining field $w$.

The first equation of (\ref{Ga.6})$_7$ is $u(0)=0$. After a
differentiation with respect to $x$ and (\ref{Ga.6})$_3$, we get
$w_z(0)=0$.

The second equation of (\ref{Ga.6})$_7$ is $w(0)=0$ and needs no treatment.

The equation (\ref{Ga.6})$_5$ can be differentiated with respect to $x$
and, thanks to (\ref{Ga.6})$_3$ leads to $w_{zz}-\beta w_{xx}=0$ at
height $z=1$.

The equation (\ref{Ga.6})$_6$ enables to compute/eliminate $\eta$.

The equation (\ref{Ga.6})$_4$ must be differentiated with respect to $t$
for $\eta$ to be replaced. Then we get
\[
\frac{w}{\beta} -p_t+\frac{2}{\RE \sqb} w_{zt} = 0.
\]
We may differentiate the previous equation with respect to $x$ so as to have
a $p_x$ term which can be replaced thanks to (\ref{Ga.6})$_1$ to have
new $u$ terms. It suffices then to differentiate this equation and use the
incompressibility (\ref{Ga.6})$_3$ to get the last condition. The full
conditions on $w$ are:
\begin{equation}
\label{Ga.9}
\begin{array}{l}
w_z(0)=0, \\
w(0)=0, \\
w_{zz}(1)-\beta w_{xx}(1)=0, \\
w_{xx}(1) -w_{ztt}(1)+\frac{3 \sqb}{\RE} w_{xxzt}(1)+\frac{1}{\RE \sqb} w_{zzzt}(1) =0 .
\end{array}
\end{equation}
The solutions (\ref{Ga.8}) will satisfy a homogeneous linear system in
the constants $C_1, C_2, C_3, C_4$. Its matrix is:

\begin{equation}
\label{Ga.9.5}
\left(
\begin{array}{cccc}
\sqb k\sinh{(\sqb k)} & -\sqb k \cosh{(\sqb k)} & \mu \sinh{\mu} & -\mu \cosh{\mu} \\
\cosh{(\sqb k)} & -\sinh{(\sqb k)} & \cosh{\mu} & -\sinh{\mu} \\
2k^2 \beta & 0 & \mu^2+\beta k^2 & 0 \\
-k^2 & \sqb k^3c^2 +\frac{2 i \beta k^4 c}{\RE} & -k^2 & \frac{2 \mu \sqb i k^3 c}{\RE}
\end{array}
\right).
\end{equation}
It suffices to compute its determinant to get the dispersion relation:

\begin{multline}
\label{Ga.10}
4\beta k^2 \mu (\beta k^2+\mu^2) +4\mu k^3 \beta^{3/2}(\mu \sinh{(k\sqb)} \sinh{\mu}-k \sqb \cosh{(k\sqb)} \cosh{\mu}) \\
-(\beta k^2 +\mu^2)^2(\mu \cosh{(k\sqb)} \cosh{\mu}-k \sqb \sinh{(k \sqb)} \sinh{\mu}) \\
-k \sqb \RE^2(\mu \sinh{(k\sqb)}  \cosh{\mu}-k\sqb \cosh{(k\sqb)} \sinh{\mu})=0.
\end{multline}
This relation is identical to the one of \cite{KM_75} except that our process of
non-dimensionnalizing makes a difference between $x$ and $z$. So
instead of $k$ (for \cite{KM_75}), we have $k \sqb$.

\subsection{Asymptotic of the phase velocity (very large Re)}
\label{subsect.a.1}
In this subsection, we prove the following Proposition:
\begin{proposition}
\label{prop1}
Under the assumptions
\begin{align}
\label{tl.1.1}
& k \sqb \RE \; c  \rightarrow + \infty, \\
\label{tl.1.2}
& k =  \; O(1), \\
\label{tl.1.3}
&\beta \rightarrow  \;0, \\
\label{tl.1.4}
&\RE \rightarrow  + \infty, \\
\label{tl.1.5}
& c =  \;O(1)\; \;  (\mbox{and } c \mbox{ bounded away from } 0),
\end{align}
if there exists a complex phase velocity $c$ solution of (\ref{Ga.10}),
then it is such that:
\begin{equation}
\label{tl.2}
c =\sqrt{\Frac{\tanh{(k \sqb)}}{k \sqb}}-\Frac{e^{i \pi
/4}\RE^{-1/2}(k \sqb)^{1/4}}{2\tanh^{3/4}{(k\sqb)}} + o(\beta^{-1/4} \RE^{-1/2}).
\end{equation}
Moreover, the decay rate in our finite-depth geometry, which is viscous, is:
\begin{equation}
\label{tl.2.5}
\mbox{Im}(kc)=\frac{-1}{2\sqrt{2}} \frac{k^{5/4}\beta^{1/8}}{\sqrt{\RE}\tanh^{3/4}{(k\sqb)}}+o(\beta^{-1/4}\RE^{-1/2}).
\end{equation}
\end{proposition}
We denote $o(f)$ (respectively $O(f)$) a function which ratio with $f$
tends to zero (respectively is bounded).

Our decay rate is not the same as Boussinesq's or Lamb's. The reason
is that our geometry is not infinite. In the regime Re$= R \,
\varepsilon^{-5/2}$ and $\beta=b \, \varepsilon$ with constant $R,b$
it gets:
\begin{equation}
\label{tl.2.75}
\mbox{Im}(kc)=\frac{-\sqrt{k}}{2\sqrt{2}\sqrt{\RE \sqb}}+o(\beta^{-1/4}\RE^{-1/2})=\frac{-\sqrt{k}\varepsilon}{2\sqrt{2}\sqrt{R \sqrt{b}}}+o(\varepsilon).
\end{equation}

Our Proposition is stated in \cite{KM_75} but not fully
proved. One must also notice that the viscosity modifies both the
real and the imaginary part of the phase velocity at the same order.

\begin{proof}
The definition of $\mu$ ($\Re (\mu) \ge 0$) and assumptions
(\ref{tl.1.1}, \ref{tl.1.2}, \ref{tl.1.3}, \ref{tl.1.5}) enable to state that $\mu^2 \rightarrow \infty$ and the
$k^2 \beta$ term tends to zero. So we have:
\begin{equation}
\label{tl.3}
\mu = e^{-i\pi/4}\sqrt{k \sqb \RE \, c} + O(\Frac{\beta^{3/4}}{\RE}),
\end{equation}
where the leading term tends to $\infty$ and its real part tends to
$+\infty$, while the error term tends to zero. As a consequence,
$\tanh \mu = 1+O(e^{-\mu})$ and $1/\cosh{\mu}= O(e^{-\mu})$. Dividing
(\ref{Ga.10}) by $\cosh{\mu}$ and using a generic notation $P(\beta ,
\mu)$ for an unspecified polynomial in $\beta, \mu$, we have:
\begin{equation}
\label{tl.4}
\begin{array}{l}
O(P(\beta, \mu )e^{-\mu}) + 4\mu k^4 \beta^2 \left(\mu \Frac{\sinh{(k
\sqb)}}{k\sqb}- \cosh{(k \sqb )} \right) \\
\hspace*{2cm} -(k^2 \beta +\mu^2 )^2 \left( \mu \cosh{(k
\sqb )}- k \sqb \sinh{(k \sqb )} \right) \\
\hspace*{2cm} - k^2 \beta \RE^2 \left( \mu
\Frac{\sinh{(k \sqb )}}{k \sqb} - \cosh{(k \sqb )} \right) = 0.
\end{array}
\end{equation}
The leading term of the second monomial is $4 k^4\beta^2 \mu^2
\sinh{(k \sqb)}/{(k \sqb)}$ while the leading term of the fourth
(last) is $-k^2\beta \RE^2 \mu \sinh{(k \sqb)} /{(k \sqb)}$. Seen the
assumptions, their ratio is $4k^2 \beta \mu \RE^{-2} =O(\beta^{5/4}
\RE^{-3/2})$. Under the assumptions (\ref{tl.1.3}, \ref{tl.1.4}), this
ratio tends to zero. So, in a first step, we can neglect the second
monomial with respect to the fourth. If we look for a non-vanishing
solution, we need to have a compensation of the only two remaining
leading terms. One may then rewrite (\ref{tl.4}) as:
\[
-(\mu^4 + hot)(\mu \cosh{(k \sqb )} +hot)-k^2\beta \RE^2 \left(\mu \frac{\sinh{(k \sqb )}}{k \sqb}+hot \right)+hot=0,
\]
where $hot$  denotes higher order terms. This reads after easy computations:

\begin{equation}
\label{tl.5}
c^2=\Frac{\tanh{(k \sqb)}}{k \sqb} +hot.
\end{equation}
Such a relation is well-known. It confirms the assumption
(\ref{tl.1.5}). To pursue the expansion we come back to (\ref{tl.4})
and expand its various monomials starting with the second:

\[
-4i k^5 \beta^{5/2}\RE\, c \frac{\sinh{(k \sqb)}}{k \sqb} + O(\beta^3)
 +O(\beta^{9/4}\RE^{1/2}).
\]
Indeed, even the leading term of this second monomial will be negligible
in comparison with $O(\beta^{7/4}\RE^{3/2})$ that we will have further. The
third monomial of (\ref{tl.4}) is more complex and we must keep:
\[
+e^{-i\pi /4}\left(k \sqb \RE \, c \right)^{5/2} \cosh{(k \sqb)} -k^3
\beta^{3/2}\RE^2 c^2 \sinh{(k \sqb)}+ O(\beta^{7/4} \RE^{3/2}).
\]
The fourth monomial of (\ref{tl.4}) is expanded:
\[
-k^2 \beta \RE^2 \left( e^{-i \pi /4} \sqrt{k \sqb \RE\, c}
 \Frac{\sinh{(k \sqb)}}{k \sqb}-\cosh{(k \sqb)} \right) +O(\beta^{7/4}
 \RE^{3/2}).
\]
Using these expansions, the equation (\ref{tl.4}) can be rewritten:
\[
\begin{array}{l}
e^{-i\pi /4}\left(k \sqb \RE \right)^{5/2} \sqrt{c} \cosh{(k \sqb)}
\left[ c^2 -\Frac{\tanh{(k \sqb)}}{k \sqb}+\Frac{e^{i\pi /4}}{\sqrt{k \sqb \RE }\sqrt{c}} \right]\\
+O(P(\beta, \mu)e^{-\mu})-k^3 \beta^{3/2}\RE^2 c^2 \sinh{(k \sqb)} + O(\beta^3)+ O(\beta^{7/4} \RE^{3/2}) = 0.
\end{array}
\]
We would like to justify that the term between square brackets
vanishes.  For that purpose, we must check that the various other
monomials are negligible in comparison with the third of those written between
the square brackets which expands into: $O((\sqb \RE )^{5/2}[(\sqb \RE
)^{-1/2}])=O(\beta \RE^2)$ if we assume (\ref{tl.1.5}). Once it is
checked (this is easy computation left to the reader), we can claim we
proved that only the terms enclosed by square brackets remain:
\begin{equation}
\label{tl.6}
c^2= \Frac{\tanh{(k \sqb)}}{k \sqb}-\Frac{e^{i \pi/4}}{\sqrt{k\sqb \RE \, c}}
+o(\beta^{-1/4}\RE^{-1/2}),
\end{equation}
and the proof is complete by computing the square root of (\ref{tl.6})
and replacing the first order of $c$ into (\ref{tl.6}) which leads to
(\ref{tl.2}).

The computation of the decay rate is straightforward.

\end{proof}

We must stress that the complex phase velocity (\ref{tl.2}) contains
two terms. The first is the classical gravitational term
($\sqrt{\tanh{(k \sqb)} /(k \sqb)}$) which may be expanded when
$\beta$ tends to zero: $1-k^2 \beta/6 + O(\beta^2)$. The second is
purely viscous and can be expanded: $-\sqrt{2}(1+i)(4 \sqrt{k})^{-1}(\RE
\, \sqb)^{-1/2}+o(\RE \, \sqb)^{-1/2})$. So the dependences of $c$
both on the gravitational and on the viscous effects are of the same
order of magnitude when $\beta$ and $(\RE \, \sqb)^{-1/2}$ are of the
same order. In this regime of very large Re, studied hereafter, the
dependence of $\RE$ on $\beta$ is such that:
\begin{equation}
\label{tl.7}
\RE \simeq \beta^{-5/2}.
\end{equation}

\subsection{Second asymptotics of the phase velocity (moderate Re)}

The definition of $\mu^2$ is $\mu^2 = k^2 \beta -ik\sqb \RE \, c$ and
we still assume a long-wave asymptotics ($\beta \rightarrow 0$). So one term
or the other dominates in $\mu^2$. The extremes are either $\mu^2 \rightarrow
\infty $ (see above) or $\mu^2 \rightarrow 0$.

In the present subsection, we investigate the latter case and
exhibit a more precise expansion than the one justified in \cite{KM_75}.
Indeed, we prove the following Proposition:

\begin{proposition}
\label{prop2}
Under the assumptions
\begin{align}
\label{tl.8}
&k  \mbox{ is }  \mbox{bounded from zero and infinity},\\
\label{tl.9}
&\mu  \rightarrow  \; 0 \mbox{ and so }\sqb \mbox{Re} \, c \rightarrow 0,\\
\label{tl.9.5}
&\beta  \rightarrow  \; 0 \mbox{ (long waves)},\\
\label{tl.9.75}
&\RE \rightarrow  +\infty,
\end{align}
if there exists a complex phase velocitiy $c$ solution of
(\ref{Ga.10}), then it is such that:
\begin{equation}
\label{tl.10}
c=-\Frac{i k \sqb \RE}{3} -\Frac{19 i k^3 \beta^{3/2} \RE^3}{90} + o(\beta^{3/2} \RE^3),
\end{equation}
and necessarily (\ref{tl.9}) implies:
\begin{equation}
\label{tl.11}
\sqb \RE \rightarrow 0,
\end{equation}
and so the phase velocity tends to zero.
\end{proposition}

Notice that if we assume $\sqb \RE \rightarrow 0$, the conclusion is
the same and the proof much simpler.

\begin{proof}
\noindent Since $\mu \rightarrow 0$, we must expand all the functions in
(\ref{Ga.10}). In this expansion, we pay special attention to the fact that $\RE
\rightarrow + \infty$ and it may not be considered as a constant
parameter of an expansion in $\beta$ (hidden in $O(\beta^2)$
as \cite{KM_75} did). After tedious expansions, there remains from (\ref{Ga.10}):
\begin{equation}
\label{tl.12}
\begin{array}{l}
O(\beta \RE^2 c^2 \mu^5) +\beta \RE^2 \mu^7(1/7!+o(\mu)) +O(\beta^{3/2}\RE \, c \mu^5)+O(\beta^2 \mu^5)\\
+O(\beta^3 \RE^2 c^2 \mu ) +O(\beta^{7/2} \RE \, c \mu ) + O(\beta^4 \RE^2 \mu )\\
+ \mu \RE^2 k^2 \beta c \left[ (c +\Frac{ik \sqb \RE }{3}) -\Frac{i k \sqb \RE \, c^2}{2}+\Frac{4 k^2 \beta \RE^2 c}{5}+2k^2 \beta c\right. \\
\hspace*{3cm} \left. +\Frac{8i k^3 \beta^{3/2} \RE}{5}+\Frac{i k^3 \beta^{3/2} \RE \, \mu^2}{3 \times 5 ! } \right]=0.
\end{array}
\end{equation}

Thanks to the assumptions (\ref{tl.8}-\ref{tl.9}) we can use that
$\sqb \RE \, c \rightarrow 0$. Then, if we denote $T_i$ the $i$th term
(among the eight) of this equation, and compare some of them, either to the first ($\mu \RE^2 k^2 \beta c^2$) or to the second ($-i\mu\RE^3 k^3 \beta^{3/2} c /3$) of the terms inside the square brackets, we have:
\begin{equation*}
\begin{array}{rlclr}
\Frac{T_1}{\mu \RE^2 k^2 \beta c^2}&=O(\mu^4), & \Frac{T_3}{-i\mu\RE^3 k^3 \beta^{3/2} c /3}&=O(\Frac{\mu^4}{\RE^2}),\\
\Frac{T_4}{T_2} &= O(\RE^{-2}), & \Frac{T_5}{\mu \RE^2 k^2 \beta c^2}&=O(\beta^2),\\
\Frac{T_6}{-i\mu\RE^3 k^3 \beta^{3/2} c /3}&=O(\Frac{\beta^2}{\RE^2}), & \Frac{T_7}{T_2}= O\left( \Frac{\beta^3}{\mu^6} \right)&=O(1).
\end{array}
\end{equation*}
As a consequence, the terms $T_1, T_3, T_4, T_5, T_6$ and $T_7$ can be
taken off from (\ref{tl.12}). Then, if we simplify by $\beta \RE^2 \mu$
and define $D$ a constant, this equation writes:
\begin{multline}
\label{tl.12.5}
\mu^6(D+O(\mu))+\left[ c^2 +\Frac{ik (\sqb \RE\, c)}{3} -c^2\Frac{i k (\sqb \RE \, c)}{2}+\Frac{4 k^2 (\sqb \RE c)^2}{5}+2k^2 \beta c^2\right. \\
\left. +\Frac{8i k^3 \beta (\sqb \RE \, c)}{5}+\Frac{i k^3 \beta\mu^2 (\sqb \RE \, c)}{3 \times 5 ! } \right]=0.
\end{multline}
Seen the assumptions (\ref{tl.8}-\ref{tl.9}), the highest order term
in the square brackets is $c^2$ which must then vanish : $c
\rightarrow 0$. Moreover, one may see that $\mu^6=O(\beta^3)+O(\sqb
\RE \, c)^3$ because of the complex definition of $\mu^2$. As a
consequence, the equation (\ref{tl.12.5}) simplifies first to
\begin{equation}
\label{tl.13}
c =-\Frac{ik \sqb \RE }{3} + o(\sqb \RE).
\end{equation}
Since we proved that $c \rightarrow 0$, so does $\sqb \RE$ as stated
in (\ref{tl.11}). Moreover, by the definition of $\mu^2$ and because of (\ref{tl.13}), one may claim
\[
\mu^2=-ik\sqb \RE\; c(1+O(1/\RE^2)) \sim -k^2 \beta \RE^2/3.
\]
In a second step, before pursuing the expansion of $c$, one may notice
that $\mu^6 =O(\beta^3 \RE^6)$ which may then be neglected in
(\ref{tl.12.5}). So, there remains only the terms in the square brackets
(simplified by $c$):
\begin{multline}
c +\Frac{ik \sqb \RE}{3} +\Frac{i k^3 \beta^{3/2} \RE^3 }{18}+o(\beta^{3/2}\RE^3)-\Frac{4 i k^3 \beta^{3/2} \RE^3 }{15}+o(\beta^{3/2}\RE^3)  \\
-\Frac{2ik^3\beta^{3/2} \RE}{3} + o(\beta^{3/2}\RE)+\Frac{8i k^3 \beta^{3/2} \RE}{5}+\Frac{i k^3 \beta^{3/2} \RE \, \mu^2}{3 \times 5 ! } = O(\beta^3)+O(\beta^3\RE^6).
\end{multline}
Among all the terms, one may justify that only the first to fourth
must be kept:
\[
c +  \frac{ik \sqb \, \RE }{3} +\frac{ik^3 \beta^{3/2} \RE^3}{18} - \frac{4i k^3 \beta^{3/2} \RE^3}{15} = o(\beta^{3/2} \, \RE^3),
\]
which gives the announced result (\ref{tl.10}) and completes the proof.
\end{proof}

Our phase velocity is different from Kakutani and Matsuuchi's
\cite{KM_75} because they assume constant Re (while it tends to
infinity) and make expansions with the other parameter.

\section{Formal derivation}

We are going to consider the influence of viscosity on the solution of the
Navier-Stokes equations in the domain $\Omega_t$. On the basis of the linear
theory of the previous section, we assume a large Re and
\begin{equation}
\label{tl.16.5}
\RE \simeq \beta^{-5/2}
\end{equation}
as announced in (\ref{tl.7}). This is the case when viscous and
gravitationnal effects balance in their influence on the variation of
the phase velocity. We further assume
\begin{equation}
\label{tl.16.75}
\alpha \sim a \varepsilon , \hspace*{2cm} \beta \sim b \varepsilon,
\end{equation}
where $\varepsilon$ is an already defined common measure of smallness
and $a,b$ are two given positive numbers. So $\alpha / \beta \simeq
1$. Our main purpose here is to derive an asymptotic system of reduced
size from the global Navier-Stokes equations in the whole {\em moving}
domain. In the inviscid case, we would derive the
classical Boussinesq system.\\

In order to prove our main result, we proceed in the same way as
\cite{KM_75} and distinguish two subdomains: the upper part ($z >
\varepsilon$) where viscosity can be neglected, and the lower part ($0
< z < \varepsilon$) which is a boundary layer at the bottom and where
viscosity must be taken into account. All the other geometrical
characteristics have already been depicted. Our first main Proposition
is stated hereafter.
\begin{proposition}
\label{Prop.1}
Let $\eta(x,t)$ be the free boundary's height. Let $u^{b,0}(x,\gamma)$
for $\gamma \in (0,+\infty)$ (resp. $u^{u,0}(x,z)$ for $z\in
(0,1+\alpha \eta(x,t))$) be the initial horizontal velocity in the
boundary layer (resp. in the upper part of the domain). If
$u^{b,0}(x,\gamma)$ is uniformly continuous in $\gamma$ and
$u^{b,0}(x,\gamma)-u^{u,0}(x,z=0)$ satisfies (\ref{L7.25}), then the
solution of the Navier-Stokes equation with this given initial
condition satisfies:
\begin{equation}
\label{Boussinesq_1}
%\left\{
\begin{array}{l}
u_t+\eta_x +\alpha u u_x-\beta \eta_{xxx}\Frac{(z^2-1)}{2} =  O(\varepsilon^2),\\
\eta_t+u_x(x,z,t)-\Frac{\beta}{2} \eta_{xxt} (z^2-\frac{1}{3})+ \alpha (u \eta)_x -\Frac{\varepsilon}{\sqrt{\pi R\sqrt{b}}} u_x \ast \Frac{1}{\sqrt{t}}   \\
+\Frac{2\varepsilon}{\sqrt{\pi}}\Int_{\gamma''=0}^{+\infty} \left(u^{b,0}_{x}(x,\gamma'') -u^{u,0}_{x}(x,z=0) \right) \Int_{\gamma'=0}^{\sqrt{\frac{\Rb}{4t}}\gamma''} e^{ -\gamma'^2}{\rm d}\gamma'{\rm d}\gamma'' =  O(\varepsilon^2),
\end{array}
%\right.
\end{equation}
where the convolution, denoted with $\ast$ is in time, the parameters
$\alpha, \beta,$ Re have been defined and $z \in (0,1+\alpha
\eta(x,t))$.

If the initial velocity is a Euler flow, then $u^{b,0}_{x}(x,\gamma'')
-u^{u,0}_{x}(x,z=0)=0$ (there is no viscous flow in the boundary
layer) and the system writes:
\begin{equation}
\label{Boussinesq_2}
\left\{
\begin{array}{ll}
u_t+\eta_x +\alpha u u_x-\beta \eta_{xxx}\Frac{(z^2-1)}{2} & =  O(\varepsilon^2),\\
\eta_t+u_x(x,z,t)-\Frac{\beta}{2} \eta_{xxt} (z^2-\frac{1}{3})+ \alpha (u \eta)_x -\Frac{\varepsilon}{\sqrt{\pi R\sqrt{b}}} u_x \ast \Frac{1}{\sqrt{t}} & =  O(\varepsilon^2),
\end{array}
\right.
\end{equation}
where the convolution is still in time.
\end{proposition}

Before starting the proof, we must justify our non-obvious choice of
method and a non-obvious term.

\begin{remark}
\label{rem_hypotheses}
  Of course, the domain in the boundary layer $\gamma \in [0,+\infty[$
  is not physical. Indeed, $u^b$ should be considered for $\gamma$
  between $\gamma=0$ and $\gamma=1$. We can extend its value up to
  $\gamma$ large (with respect to 1), but small (with respect to
  $1/\varepsilon$ so as to ensure $z=\varepsilon \gamma <1$). For
  instance, one may choose $\gamma=1/\sqrt{\varepsilon}$ (equivalently
  $z=\sqrt{\varepsilon}$) or any value between $\gamma = 1$ and
  $\gamma=+\infty$ such that $z=\varepsilon \gamma \ll 1$. 

  The same applies in the upper part. Indeed, $u^u(x,z,t)$ should be
  considered for $z \in (\varepsilon,1)$ and $u^u(x,z=0,t)$ should be
  $u^u(x,z=\varepsilon,t)$.

  One can then write the boundary condition at any height like
  $\gamma=1/\sqrt{\varepsilon}$ and force that the final result does
  not rely on this choice.

  As is classical in boundary layer analysis, these more justified
  notations would give the same result as our choice. So we will use
  the most straightforward and consider $u^u$ for $z \in (0,1+\alpha
  \eta(x,t))$ and $u^b$ for $\gamma \in (0,+\infty)$.
\end{remark}

\begin{remark}
\label{remark_Euler_NS}
  The double integral term in (\ref{Boussinesq_1}) is new and
  surprising because of its dependence on the initial condition. One
  could wonder whether assuming vanishing initial conditions in the
  boundary ($u^{b,0}=u^{u,0}$ or equivalently that the initial flow is
  of Euler type), that would greatly simplify the computations, would
  be physical. A physical question is then to know whether an initial
  (inviscid) flow in the boundary layer (where Navier-Stokes applies)
  establishes (as a Navier-Stokes flow) fast or not.

  We claim that the characteristic time for the viscous effects to
  appear is roughly $T_{NSE}=\rho h_0^2/\nu$ or $T_{NSE}=\rho
  l^2/\nu$. Then, its ratio with the characteristic time of the
  inviscid gravity flow ($l/c_0$) is either Re
  $\sqrt{\beta}=\varepsilon^{-2}$ or Re $/
  \sqrt{\beta}=\varepsilon^{-3}$ respectively. In any case, it is
  large and the flow in the boundary layer does not establish fast
  enough. It does not enable to claim that a Euler initial condition
  is physically compatible with Navier-Stokes equations for moderate
  times. Khabakhpashev \cite{Khabakhpashev_87} already discussed it
  although he started from rest !
\end{remark}

In the first subsection \ref{subsec3.1} we treat the upper part where
convenient equations of (\ref{Ga.5}) are kept. Then in subsection
\ref{subsec3.2}, after a rescaling, we solve in the boundary layer the
convenient equations extracted from (\ref{Ga.5}). The solutions are
forced to match through a continuity condition at the boundary
($z=\varepsilon$) discussed in Subsection \ref{subsec3.3}. At this
stage, the system still has $u^u(1)$ and $\int_0^1 u^u$ terms. So
Subsection \ref{subsec3.4} is devoted to making explicit and simple
the dependence on $z$ so as to get rid of these extra terms.

\subsection{Resolution in the upper part}
\label{subsec3.1}

The upper part is characterized by $\varepsilon < z <1+\alpha
\eta(x,t)$ and $x,t \in \mathbb{R}$. We start from the system for the
fields in the upper part and write $u, w, p$ instead of $u^u, w^u,
p^u$ for the sake of simplification. The height of the perturbation
$\eta$ is only defined in the upper part and so will always be denoted
the same in the boundary layer. The system of PDE in the upper part is
extracted from (\ref{Ga.5}):
\begin{equation}
\label{u.1}
\left\{
\begin{array}{ll}
u_t+\alpha u u_x + \Frac{\alpha}{\beta} w u_z-\Frac{\sqrt{\beta}}{\mbox{Re}} u_{xx} -\Frac{1}{\mbox{Re} \sqrt{\beta}}u_{zz} +p_x =0, & \varepsilon < z < 1+\alpha \eta, \\
w_t +\alpha u w_x + \Frac{\alpha}{\beta}w w_z -\Frac{\sqrt{\beta}}{\mbox{Re}}w_{xx} -\Frac{1}{\mbox{Re }\sqrt{\beta}}w_{zz} +p_z =0, & \varepsilon < z < 1+\alpha \eta,  \\
\beta u_x +w_z =0,&  \varepsilon < z < 1+\alpha \eta, \\
-\alpha \eta_x (\eta -p)+\frac{1}{\RE}(-2 \alpha \sqb u_x \eta_x+ (u_z+w_x) )= 0& \mbox{on } z=1+\alpha \eta, \\
\eta -p +\frac{1}{\RE} (-\alpha\eta_x(u_z+w_x)-2\sqb u_x)=0& \mbox{on } z=1+\alpha \eta, \\
\eta_t+\alpha u \eta_x-\frac{1}{\beta} w = 0 & \mbox{on } z=1+\alpha \eta.
\end{array}
\right.
\end{equation}

Since we assume Re $\simeq \varepsilon^{-5/2}$, the terms $ \sqb /
\RE$ are of the order of $\varepsilon^3$ and the terms $1/ (\RE \, \sqb)$
of the order of $ \varepsilon^{2}$. This simplifies (\ref{u.1})$_1$
and (\ref{u.1})$_2$ and justifies to take off the Laplacian. As a
consequence, we must not keep the two dynamic conditions
(\ref{u.1})$_4$ and (\ref{u.1})$_5$ since they are associated to a
Laplacian. We decide to drop (\ref{u.1})$_4$.

Alternatively, one can stress that (\ref{u.1})$_5$ gives $\eta
-p=O(\varepsilon^3)$ and so the lhs of (\ref{u.1})$_4$ is
$O(\varepsilon^4)+O(\varepsilon^{5/2})$. Since we expand until the
order two, one may claim the equation reduces to $0=0$. But one could
also simplify by $1/ \RE$ ($ \simeq \varepsilon^{5/2}$) and be driven
to a new equation. This equation would provide one more condition to
the two equations for two fields. It is not surprising to see that the
final solution would then be $u=0$. The error is that we must drop one
boundary condition unless we have one additionnal condition. The above
argument to get rid of (\ref{u.1})$_4$ is sufficient.

On this topic, the literature uses the same equations, but the
argument for dropping one boundary condition is rarely explicited. In
\cite{KM_75}, Kakutani and Matsuuchi claim ``the condition
[(\ref{u.1})$_4$] is automatically satisfied'' (p. 242 al. 3) which is
either wrong (the equation disappears) or incomplete (what if they
simplify by Re $=\varepsilon^{5/2}$ ?).

%{\bf They also claim `` the boundary conditions at
%the free surface may be replaced by the inviscid ones so far as the
%present order of approximation is concerned '' (p. 240 col. 2
%al.1). The replacement does not depend on the `` present order of
%approximation '', but on the occurence of the Laplacian or not.}

In \cite{Dutykh_Dias_CRAS}, Dutykh and Dias solve the same problem and
write two equations (their (3) and (4)) among which they keep
only one for the derivation without explaining this drop.\\

Let us come back to the resolution in the upper part. The equation
(\ref{u.1})$_3$ gives $w$ up to a constant that can be found in
(\ref{u.1})$_6$:
\begin{equation}
\label{u.2}
w(x,z,t)= -\beta\int_{1+\alpha \eta}^z u_x(x,z',t) \,{\rm d}z'+\beta(\eta_t+\alpha u(1+\alpha\eta)\eta_x),
\end{equation}
and we stress that this equation is exact. For the computations later,
we need to expand this equation up to the third order:
\begin{equation}
\label{u.2p}
w(x,z,t)=\beta(\eta_t+\int_0^1 u_x)-\beta\int_0^z u_x + \alpha \beta (u(1) \eta)_x + O(\varepsilon^3).
\end{equation}
The second order of the previous equations suffices to
determine $p$ from (\ref{u.1})$_2$ up to a constant:
\begin{align}
p(x,z,t) = &p(x,1+\alpha \eta,t)-\beta(\eta_{tt}+\int_0^1 u_{xt})(z-1)\nonumber \\
           & +\beta \int_1^z \int_0^{z'} u_{xt}(x,z'',t) \,{\rm d}z'' \,{\rm d}z'+O(\varepsilon^2).
\end{align}
Thanks to (\ref{u.1})$_5$ the constant may be found ($p(1+\alpha \eta)
=\eta +O(\varepsilon^3)$) and so:
%
%\begin{equation}
\begin{align}
%\begin{array}{l}
\nonumber p(x,z,t) = & \eta -\beta(\eta_{tt}+\int_0^1 u_{xt})(z-1)+\beta \int_1^z \int_0^{z'} u_{xt}(x,z'',t) \,{\rm d}z'' \,{\rm d}z'+O(\varepsilon^2)\\
\label{u.3}
%\hspace*{1.36cm}= \eta -\beta\eta_{tt}(z-1) +\beta \int_1^z \int_1^{z'} u_{xt}(x,z'',t) \,{\rm d}z'' \,{\rm d}z'+O(\varepsilon^2).
= &\eta -\beta\eta_{tt}(z-1) +\beta \int_1^z \int_1^{z'} u_{xt}(x,z'',t) \,{\rm d}z'' \,{\rm d}z'+O(\varepsilon^2).
\end{align}
%\end{array}
%\end{equation}
%
Then the remaining field $u$ satisfies (\ref{u.1})$_1$ at
the first order:
\begin{equation}
\label{u.4}
u_t+\eta_x +\alpha u u_x+\alpha u_z(\eta_t +\int_z^1u_x)-\beta \eta_{xtt}(z-1) -\beta \eta_{xxx}(z-1)^2 /2 = O(\varepsilon^2),
\end{equation}
where we have replaced the $u_{xxt}$ by $-\eta_{xxx}$ as usual.\\
We still have to solve the equations in the lower part.

\subsection{Resolution in the boundary layer}
\label{subsec3.2} 

We need first to recall some classical properties of Laplace
transforms.

\subsubsection{Some useful properties}

Before solving the equations in the lower part, we list here some
classical properties of the Laplace transform. We start from the
definition

\begin{equation}
\label{L4}
\mathcal{L} (f) (p) = \hat{f}(p) =\int_{t \in \mathbb{R}^+} f(t) e^{-p t} \, {\rm d}t.
\end{equation}
It is well-known that the Laplace transform of the derivative is given by
\begin{equation}
\label{L5}
\mathcal{L} (f')(p)=-f(0)+p \mathcal{L} (f) (p).
\end{equation}
If the two transforms $\mathcal{L} (f)(p)$ and $\mathcal{L} (g)(p)$
converge absolutely for $p=p_0$, and if both $f$ and $g$ are
absolutely integrable and bounded in every finite interval that does
not include the origin such as $(p_1,p_2)$ where $0<p_1 \leq p_2$,
then the Laplace transform of the convolution exists for $p$ such that
$\Re(p) \geq \Re(p_0)$ (\cite{Doetsch} Th. 10.1), even converges
absolutely, and satisfies:
\begin{equation}
\label{L6}
\mathcal{L} (f)(p) \mathcal{L} (g)(p)= \mathcal{L} (f \ast g)(p).
\end{equation}
Below, we use the following definition of the convolution,
linked to the Laplace transform:
\begin{equation}
\label{L7}
f_1 \ast f_2 (t)=\Int_0^t f_1(u)f_2(t-u) {\rm d} u.
\end{equation}
These formulas will be useful in the next subsection.

\subsubsection{The fields in the boundary layer}

The lower part of the domain ($0<z < \varepsilon$) is a boundary
layer. We start from the system for the bottom fields, written $u, w,
p$ instead of $u^b, w^b, p^b$ for the sake of simplification and
extracted from (\ref{Ga.5}):
\begin{equation}
\label{l.1}
\left\{
\begin{array}{ll}
u_t+\alpha u u_x + \Frac{\alpha}{\beta} w u_z-\Frac{\sqrt{\beta}}{\mbox{Re}} u_{xx} -\Frac{1}{\mbox{Re }\sqrt{\beta}}u_{zz} +p_x =0 & \mbox{ for } 0< z < \varepsilon , \\
w_t +\alpha u w_x + \Frac{\alpha}{\beta}w w_z -\Frac{\sqrt{\beta}}{\mbox{Re}}w_{xx} -\Frac{1}{\mbox{Re }\sqrt{\beta}}w_{zz} +p_z =0 & \mbox{ for } 0< z < \varepsilon , \\
\beta u_x +w_z =0& \mbox{ for } 0< z < \varepsilon , \\
u(z=0)= 0 \; \mbox{ and } w(z=0) = 0. & 
\end{array}
\right.
\end{equation}

As is justified in subsection \ref{subsect.a.1}, the viscous and
gravitational effects balance when $\RE \simeq \beta^{-5/2}$ (same as
(\ref{tl.7})). So we remind the reader of the assumptions $\RE = R \;
\varepsilon^{-5/2}$, $\alpha=a \varepsilon$ and $\beta = b
\varepsilon$ for constant $R, a, b$. We are naturally led to change
the scale in $z$ as in any boundary layer. Let us introduce a new
vertical variable $\gamma =z / \varepsilon$. The new fields should be
denoted in another way. Nevertheless, we will not change the
notation for the sake of simplification. The new system writes:

\begin{equation}
\label{l.2}
\left\{
\begin{array}{l}
u_t+\alpha u u_x +\frac{a}{b \varepsilon} w u_{\gamma} -\frac{\sqrt{b}}{R}\varepsilon^{3} u_{xx} -\frac{u_{\gamma \gamma}}{R \, \sqrt{b}} + p_x = 0,\\
w_t+\alpha u w_x +\frac{a}{b \varepsilon} w w_{\gamma} -\frac{\sqrt{b}}{R}\varepsilon^{3} w_{xx} -\frac{w_{\gamma \gamma}}{R \, \sqrt{b}} + \frac{p_{\gamma}}{\varepsilon} = 0,\\
\varepsilon \beta u_x + w_{\gamma} = 0,\\
u(\gamma =0)= 0 \; \mbox{ and } w(\gamma =0) = 0.
\end{array}
\right.
\end{equation}
One must notice that the Laplacian lets some remaining terms of zeroth
degree in this system. So the viscosity is relevant in the boundary
layer.

\noindent We can find the vertical velocity
from (\ref{l.2})$_3$ and (\ref{l.2})$_4$:
\begin{equation}
\label{l.3}
w(x,\gamma,t)=-\varepsilon \beta \int_0^{\gamma}u_x(x,\gamma',t) \, {\rm d} \gamma'.
\end{equation}
Carrying backward the previous equation in (\ref{l.2})$_2$, one has
$p_{\gamma} = O(\varepsilon^3)$. So as to determine $p$, we need to
use the continuity relation for the pressure
($p(x,\gamma=1,t)=p^u(x,z=\varepsilon,t)$) unless we cannot go on. Since
we know the pressure in the upper part $p^u$ from (\ref{u.3}), we can write:
\begin{equation}
\label{l.4}
p(x,\gamma,t)=p(x,\gamma = 1,t)+O(\varepsilon^3)=p^u(x, \varepsilon,t)+O(\varepsilon^3)=\eta(x,t)+O(\varepsilon).
\end{equation}
Using this equation and (\ref{l.3}) in (\ref{l.2})$_1$, we have at
zeroth order:
\begin{equation}
\label{L1}
u_t+\eta_x-\frac{u_{\gamma \gamma}}{R \sqrt{b}} = O(\varepsilon).
\end{equation}
This equation must be completed with initial condition
\begin{equation}
\label{L2}
u(x,\gamma,t=0)=u^{b,0}(x,\gamma),
\end{equation}
and boundary condition:
\begin{equation}
\label{L3}
\left\{ \begin{array}{rl}
u(x,\gamma=0,t) = & 0,\\
u(x,\gamma \rightarrow +\infty,t)= &u^{u}(x,z=0,t) \mbox{ (continuity condition)}.
\end{array}\right.
\end{equation}
These are the equations to be solved.

Since we solve a Cauchy problem for a heat-like equation, we have an
initial condition and so we must use the time-Laplace
transform. In \cite{KM_75}, the authors do not take an initial
condition, and uses a time-Fourier transform. In all his articles,
P.L. Liu, and coauthors ({\em e.g.}  \cite{Liu_Orfila_04}), quote
\cite{Mei_95} (pp. 153--159) in which a sine-tranform (in $\gamma$) is
used, but the initial condition is set to zero. In a separate
calculation, not reproduced here, we used the same sine-transform in
$\gamma$ and paid attention to the initial condition. We were led to
the very same result as the one stated hereafter.

We solve the system (\ref{L1}-\ref{L3}) in the following Lemma.

\begin{lemma}
\label{b.1}
If the initial conditions $u^{b,0}(x,\gamma)$ and $u^{u,0}(x,z=0)$ are
uniformly continuous in $\gamma$ and satisfy 
\begin{equation}
\label{L7.25}
\begin{array}{c}
\Int_0^{\infty} \mid u^{b,0}(x,\gamma)-u^{u,0}(x,z=0) \mid {\rm d} \gamma < \infty,\\
\Int_0^{\infty} \mid u^{b,0}_x(x,\gamma)-u^{u,0}_x(x,z=0) \mid {\rm d} \gamma < \infty,
\end{array}
\end{equation}
then the solution to (\ref{L1}-\ref{L3}) is 
\begin{equation}
\label{L18}
\begin{array}{rl}
u(x,\gamma,t)= & u^u(x,z=0,t)+\frac{\srb}{2}\int_{0}^{+\infty} f_{0}(x,\gamma') \frac{e^{-\frac{\Rb(\gamma'-\gamma)^2}{4t}}}{\sqrt{\pi t}} {\rm d} \gamma'\\
  & -u^u(x,0,.) \ast \mathcal{L}^{-1}(e^{-\sigma \gamma}) \\
  & -\frac{\srb}{2}\int_{0}^{+\infty} f_{0}(x,\gamma') \frac{e^{\frac{-\Rb(\gamma'+\gamma)^2}{4t}}}{\sqrt{\pi t}} {\rm d} \gamma' +O(\varepsilon),
\end{array}
\end{equation}
where $f_0(x,\gamma)=u^{b,0}(x,\gamma)-u^{u,0}(x,z=0)$, $u^u$ is the
horizontal velocity in the upper part that satisfies (\ref{u.4}) and
$\sigma$ is the only root with a positive real part of $R\sqrt{b} \,
p$:
\begin{equation}
\label{L13}
\sigma=\sigma(p)=\sqrt{R \sqrt{b} p}.
%= \sqrt{R \sqrt{b}} \; e^{i\, \mbox{\footnotesize{sgn}}(\xi) \, \frac{\pi}{4} }\sqrt{\mid \xi \mid },
\end{equation}
where $p$ is the dual variable of time $t$ and the convolution is in time.
\end{lemma}

\begin{remark}
\label{rem1}
The solution of (\ref{L1}) may be known only up to any function of
$x$. The boundary condition (\ref{L3}) enables to determine this
function.
\end{remark}

\begin{remark}
\label{rem2}
The compatibility of the conditions (\ref{L2}) and (\ref{L3}) forces
to have, when $\gamma$ tends to $+\infty$:
\[
u^{b,0}(x,\gamma) \rightarrow u^{u,0}(x,z=0),
\]
and, when $\gamma \rightarrow 0$:
\[
u^{b,0}(x,\gamma=0)=0.
\]
\end{remark}
Meanwhile we also prove the following Proposition
\begin{proposition}
\label{Prop.2}
Under the same assumptions as in Lemma \ref{b.1}, the bottom shear stress is
\begin{equation}
\label{bot.shear}
\begin{array}{rl}
\tau^{b}=\left( \DP{u^b}{\gamma} \right)_{\gamma=0}=&\Frac{\sqrt{R\sqrt{b}}u^{u}(x,z=0,0)}{\sqrt{\pi}} {\rm p.v. }\Frac{1}{\sqrt{t}} \\
  & +\Frac{\sqrt{R\sqrt{b}}}{\sqrt{\pi}} \Int_0^t \Frac{u^{u}_t(x,z=0,t-s)}{\sqrt{s}} {\rm d}s,
\end{array}
\end{equation}
where {\rm p.v.} denotes the principal value as defined in the theory
of distributions.
\end{proposition}

First let us prove Proposition \ref{Prop.2}.
\begin{proof} The initial condition $f_0$ may not make any difference
  (it can be seen through an explicit computation), so the
  correspondig term is taken off. Then a simple differentiation with
  respect to $\gamma$ and the following formula (See \cite{Doetsch}
  p. 320)
\[
\mathcal{L}^{-1}\left(e^{-a\sqrt{p}}\right) =\Frac{a}{2\sqrt{\pi}t^{3/2}} e^{-\frac{a^2}{4t}},
\]
applied to (\ref{L18}) for any positive $a$ leads to
\begin{align*}
\tau^{b}= & -\Frac{{\rm d}}{{\rm d}\gamma}\left( \Int_0^t u^{u}(x,z=0,t-s)\Frac{e^{-\frac{R\sqrt{b}\gamma^2}{4s}}\sqrt{R\sqrt{b}}\gamma}{2\sqrt{\pi} s^{3/2}} {\rm d}s \right)+O(\varepsilon)\\
 = & -\sqrt{R\sqrt{b}}\Int_0^t \Frac{u^{u}(x,z=0,t-s)}{2\sqrt{\pi} s^{3/2}}e^{-\frac{R\sqrt{b}\gamma^2}{4s}}{\rm d}s \\
 & -\Frac{\sqrt{R\sqrt{b}}}{\sqrt{\pi}}\Int_0^t \Frac{u^{u}(x,z=0,t-s)}{ s^{1/2}}\left(-\Frac{R\sqrt{b}\gamma^2}{4s^2}e^{-\frac{R\sqrt{b}\gamma^2}{4s}}\right){\rm d}s+O(\varepsilon).
\end{align*}
The second term may be integrated by parts to get
\[
\begin{array}{l}
-\Frac{\sqrt{R\sqrt{b}}}{\sqrt{\pi}} \left(\Frac{u^{u}(x,z=0,0)}{\sqrt{t}} e^{-\frac{R\sqrt{b}\gamma^2}{4t}}\right.\\
\hspace*{1.5cm}\left.-\Int_0^t \left(-\Frac{u^{u}_t(x,z=0,t-s)}{\sqrt{s}}-\Frac{u^{u}(x,z=0,t-s)}{2 s^{3/2}}  \right)e^{-\frac{R\sqrt{b}\gamma^2}{4s}}{\rm d}s\right),
\end{array}
\]
which simplifies partially with the first term. At the end, there remains only
\[
\Frac{\sqrt{R\sqrt{b}}}{\sqrt{\pi}} \Frac{u^{u}(x,z=0,0)}{\sqrt{t}} e^{-\frac{R\sqrt{b}\gamma^2}{4t}}+\Frac{\sqrt{R\sqrt{b}}}{\sqrt{\pi}} \Int_0^t \Frac{u^{u}_t(x,z=0,t-s)}{\sqrt{s}} e^{-\frac{R\sqrt{b}\gamma^2}{4s}}{\rm d}s.
\]
This justifies the formula as is classical in the theory of distributions.
\end{proof}

The scheme of the proof of Lemma \ref{b.1} is to solve (\ref{L1}) up
to two unknown functions, then to determine these functions so as to
satisfy the initial and boundary conditions. This provides a necessary
formula. We check in Appendix \ref{appendix_A} that the solution satisfies
the boundary and initial conditions. Let us prove Lemma \ref{b.1}.
\begin{proof}

\noindent Let us denote
\begin{equation}
\label{L8}
f(x,\gamma,t)= u(x,\gamma,t)-u^u(x,z=0,t).
\end{equation}
Since $f_t=u_t+\eta_x+O(\varepsilon)$ (thanks to (\ref{u.4})) 
and $f_{\gamma}=u_{\gamma}$, the equation (\ref{L1}) writes:
\begin{equation}
\label{L9}
f_t-f_{\gamma \gamma}/(R\sqrt{b}) = O(\varepsilon).
\end{equation}
The initial condition is
\begin{equation}
\label{L10}
f(x,\gamma,t=0)=u^{b,0}(x,\gamma)-u^{u,0}(x,z=0)=:f_0(x,\gamma),
\end{equation}
and the boundary conditions read
\begin{equation}
\label{L11}
\begin{array}{l}
f(x,\gamma=0,t)=-u^u(x,0,t),\\
\lim_{\gamma \rightarrow +\infty} \lim_{\varepsilon \rightarrow 0} u(x,\gamma,t)-u^{u}(x,z=\varepsilon,t)= \lim_{\gamma \rightarrow +\infty} f(x,\gamma,t)=0.
\end{array}
\end{equation}
The second condition is merely the continuity condition of the
horizontal velocity at the border of the boundary layer. So we are
driven to a heat equation in a half space with vanishing condition at
infinity, and non-homogeneous initial and bottom conditions. Through a
Laplace transform in time, denoted either ${\mathcal L}(f)$ or
$\hat{f}$, (\ref{L9}) becomes
\begin{equation}
\label{L12}
-f_0(x,\gamma)+p\hat{f}(p)-\Frac{\hat{f}_{\gamma \gamma}}{\Rb} = O(\varepsilon).
\end{equation}
In order to solve this non-homogeneous ODE, we start with the
homogeneous one and recall that we define $\sigma$ as the only root
with a positive real part of $R \sqrt{b}p$ in (\ref{L13}). Its
solutions are
\[
\hat{f}(x,\gamma,p)=C_1(x,p) e^{+\sigma \gamma}+C_2(x,p)e^{-\sigma \gamma}+O(\varepsilon).
\]
By applying the method of parameters variation, we look for
$C_1(x,\gamma,p)$, $C_2(x,\gamma,p)$ such that:
\[
-C_{1,\gamma}\sigma e^{\sigma \gamma}+C_{2,\gamma}\sigma e^{-\sigma \gamma}=\Rb f_0(x,\gamma) + O(\varepsilon),
\]
and solving (\ref{L12}) amounts to solving the system of two equations
with two unknown functions $C_1$ and $C_2$:
\[
\left\{ \begin{array}{rl}
C_{1,\gamma}e^{\sigma \gamma}+C_{2,\gamma} e^{-\sigma \gamma} = & 0, \\
-C_{1,\gamma}e^{\sigma \gamma}+C_{2,\gamma} e^{-\sigma \gamma} = & \Frac{\Rb}{\sigma} f_0, 
\end{array}\right.
\]
which solution is (thanks to assumption (\ref{L7.25})):
\[
\left\{ \begin{array}{rl}
C_1(x,\gamma,p) = & -\Frac{\Rb}{2 \sigma} \Int_{+\infty}^{\gamma} f_{0}(x,\gamma')e^{-\sigma \gamma'} {\rm d} \gamma'+\tilde{C}_1(x,p), \\
C_2(x,\gamma,p) = & +\Frac{\Rb}{2 \sigma} \Int_{0}^{\gamma} f_{0}(x,\gamma')e^{\sigma \gamma'} {\rm d} \gamma'+\tilde{C}_2(x,p).
\end{array}
\right.
\]
So, the full solution is
\[
\begin{array}{rl}
\hat{f}(x,\gamma,p)=& -\Frac{\Rb}{2 \sigma} \Int_{+\infty}^{\gamma} f_{0}(x,\gamma')e^{-\sigma \gamma'} {\rm d} \gamma' e^{+\sigma \gamma}+ \tilde{C}_1(x,p)e^{+\sigma \gamma}\\
 & +\Frac{\Rb}{2 \sigma} \Int_{0}^{\gamma} f_{0}(x,\gamma')e^{\sigma \gamma'} {\rm d} \gamma'e^{-\sigma \gamma}+\tilde{C}_2(x,p) e^{-\sigma \gamma}+O(\varepsilon).
\end{array}
\]
We look for $ \tilde{C}_1$ first. Since $f_0$ is bounded, simple bounds prove that the first, third and fourth terms are bounded. So
\[
\tilde{C}_1(x,p)=0.
\]
The unknown function $\tilde{C}_2(x,p)$ is then given by the boundary
condition (\ref{L11})$_1$ at the bottom:
\[
\tilde{C}_2(x,p)=-u^u(x,z=0,p)-\Frac{\Rb}{2 \sigma} \Int_0^{+\infty} f_{0}(x,\gamma')e^{-\sigma \gamma'} {\rm d} \gamma'.
\]
In a necessary way,
\begin{equation}
\label{L17}
\begin{array}{rl}
\hat{f}(x,\gamma, p)= & +\Frac{\Rb}{2 \sigma} \Int^{+\infty}_{\gamma} f_{0}(x,\gamma')e^{-\sigma \gamma'} {\rm d} \gamma' e^{+\sigma \gamma} \\
 & +\Frac{\Rb}{2 \sigma} \Int_{0}^{\gamma} f_{0}(x,\gamma')e^{\sigma \gamma'} {\rm d} \gamma' e^{-\sigma \gamma}\\
 & -\left(\hat{u}^u(x,z=0,p)+\Frac{\Rb}{2 \sigma} \Int_0^{+\infty} f_{0}(x,\gamma')e^{-\sigma \gamma'} {\rm d} \gamma' \right) e^{-\sigma \gamma}+O(\varepsilon).
\end{array}
\end{equation}
From the definition of $f$, the existence of an inverse Laplace
transform and formula (\ref{L6}), one knows that:
\begin{align*}
f(x,\gamma,t)=& \Frac{\Rb}{2} \Int^{+\infty}_{\gamma} f_{0}(x,\gamma')\mathcal{L}^{-1}\left(\frac{e^{-\sigma (\gamma'-\gamma)}}{\sigma}\right) {\rm d} \gamma' \\
& +\Frac{\Rb}{2} \Int_{0}^{\gamma} f_{0}(x,\gamma') \mathcal{L}^{-1}\left(\frac{e^{\sigma (\gamma'-\gamma)}}{\sigma}\right) {\rm d} \gamma' \\
& -u^u(x,z=0,.) \ast \mathcal{L}^{-1}\left(e^{-\sigma \gamma} \right)\\
& -\Frac{\Rb}{2} \Int_0^{+\infty} f_{0}(x,\gamma')\mathcal{L}^{-1}\left(\Frac{e^{-\sigma (\gamma' + \gamma)}}{\sigma}\right) {\rm d} \gamma'+O(\varepsilon).
\end{align*}
Owing to the formula (see \cite{Doetsch})
\[
\mathcal{L}^{-1} \left( \Frac{e^{-\tilde{a}\sqrt{p}}}{\sqrt{p}} \right) =\Frac{1}{\sqrt{ \pi t}} e^{-\frac{\tilde{a}^2}{4t}},
\]
if $\tilde{a}>0$, one may justify the explicit form of $u$ given
in (\ref{L18}). Until the end of this article, we denote the function
of time $t$:
\begin{equation}
\label{L17.5}
A=A(t)=\sqrt{\Frac{\Rb}{4t}}.
\end{equation}
We still have to check that the initial condition (\ref{L10}) and
remaining of the boundary conditions (\ref{L11})$_2$ are satisfied by
$u$ given by (\ref{L18}). This is completed in Appendix \ref{appendix_A}.

So we completed the proof of the whole Lemma \ref{b.1}.
\end{proof}

From (\ref{l.3}) and (\ref{L18}), we can then compute the vertical velocity
\begin{equation}
\label{L19}
\begin{array}{rl}
w^b(x,\gamma,t) =&-\varepsilon \beta \int_0^{\gamma}u_x^b(x,\gamma',t) {\rm d} \gamma'\\
 =&-\varepsilon \beta u_x^u(x,0,t)\gamma +\varepsilon \beta u_x^u(x,0,.) \ast \mathcal{L}^{-1}\left(\Frac{e^{-\sigma \gamma}-1}{-\sigma}\right) \\
  & -\varepsilon \beta\Frac{A(t)}{\sqrt{\pi}}\Int_{\gamma'=0}^{\gamma} \Int_{\gamma''=0}^{+\infty} f_{0,x}(x,\gamma'')e^{-A^2(\gamma''-\gamma')^2}\, {\rm d}\gamma''  {\rm d}\gamma'\\
  &  +\varepsilon \beta\Frac{A(t)}{\sqrt{\pi}}\Int_{\gamma'=0}^{\gamma} \Int_{\gamma''=0}^{+\infty} f_{0,x}(x,\gamma'')e^{-A^2(\gamma''+\gamma')^2}\, {\rm d}\gamma''  {\rm d}\gamma'+O(\varepsilon^2 \beta).
\end{array}
\end{equation}
We still have to satisfy the continuity conditions of all the fields
$u,w,p$.
\subsection{Continuity conditions}
\label{subsec3.3}
In the present subsection, we need to write explicitly the
superscripts $u$ and $b$ for the upper part and bottom regions
respectively. We write the computed fields at the same height
$\varepsilon$ that is the common frontier of both subdomains.

We already used the continuity of pressure that led us to
(\ref{l.4}). So the pressure is continuous.\\

Regarding the horizontal velocity, we must notice that the limit when
$\gamma \rightarrow +\infty$ of $\lim_{\varepsilon \rightarrow
  0}(u^b(x,\gamma,t)-u^u(x, \varepsilon \gamma,t))=f(x,\gamma,t)$ has
already been computed as vanishing (see Appendix \ref{appendix_A}). So
the boundary condition (\ref{L11})$_2$ is already satisfied and the
horizontal velocity is continuous.\\

Concerning the vertical velocity, we can use the velocity in the upper
part $w^u$ from (\ref{u.2p}) expanded in $\varepsilon$:
\[
\begin{array}{rl}
w^u(x,\varepsilon \gamma,t)=& \beta(\eta_t+\int_0^1 u^u_x)-\beta\int_0^{\varepsilon \gamma} u^u_x + \alpha \beta (u^u(1) \eta)_x + O(\varepsilon^3)\\
 = & \beta(\eta_t+\int_0^1 u^u_x)-\beta \varepsilon \gamma u^u_x(z=0) + \alpha \beta (u^u(1) \eta)_x + O(\varepsilon^3).
\end{array}
\]
One may notice that as anywhere else, the $u^u(z=0)$ could be replaced
by $u^u(z=\varepsilon)$ and $\int_0^1 u^u_x$ by
$\int_{\varepsilon}^{1+\alpha \eta}u^u_x$ and so on. The formula would
be the same and the final result would be the same.

The velocity in the bottom $w^b$ is given in (\ref{L19}).  The
difference $w^u(x,\varepsilon \gamma, t)-w^b(x,\gamma,t)$ can be
expanded in $\varepsilon$:
\begin{align}
\nonumber
 w^u -w^b= & \beta(\eta_t+\int_0^1 u^u_x)-\beta \varepsilon \gamma u^u_x(z=0)+ \alpha \beta (u^u(1) \eta)_x + O(\varepsilon^3)\\
\nonumber
         & +\varepsilon \beta u^u_x(x,z=0,t)\gamma-\varepsilon \beta u_x^u(x,0,.) \ast \mathcal{L}^{-1} \left(\Frac{e^{-\sigma \gamma}-1}{-\sigma}\right)\\
\nonumber
         & +\varepsilon \beta \Frac{A}{\sqrt{\pi}} \Int_{\gamma'=0}^{\gamma}\Int_{\gamma''=0}^{+\infty} f_{0,x}(x,\gamma'')e^{-A^2(\gamma''-\gamma')^2} {\rm d}\gamma''  {\rm d}\gamma' \\
\nonumber
         & - \varepsilon \beta \Frac{A}{\sqrt{\pi}} \Int_{\gamma'=0}^{\gamma} \Int_{\gamma''=0}^{+\infty}f_{0,x}(x,\gamma'')e^{-A^2(\gamma''+\gamma')^2} {\rm d}\gamma''  {\rm d}\gamma' +O(\varepsilon^3)\\
\nonumber
        = & \beta(\eta_t+\int_0^1 u^u_x)+ \alpha \beta (u^u(1) \eta)_x-\varepsilon \beta u_x^u(x,0,.) \ast \mathcal{L}^{-1} \left(\Frac{1}{\sigma}\right)\\
\label{l.17}
          & +\varepsilon \beta \Frac{A}{\sqrt{\pi}} \Int_{\gamma'=0}^{\gamma}\Int_{\gamma''=0}^{+\infty} f_{0,x}(x,\gamma'')e^{-A^2(\gamma''-\gamma')^2} {\rm d}\gamma''  {\rm d}\gamma'\\
\nonumber
          & -\varepsilon \beta \Frac{A}{\sqrt{\pi}} \Int_{\gamma'=0}^{\gamma} \Int_{\gamma''=0}^{+\infty}f_{0,x}(x,\gamma'')e^{-A^2(\gamma''+\gamma')^2} {\rm d}\gamma''  {\rm d}\gamma' +O(\varepsilon^3),
\end{align}
up to functions that tend exponentially to zero when $\gamma
\rightarrow +\infty$.\\

We still must simplify the two last double integrals. This is made in
the following Lemma
\begin{lemma}
\label{b.2}
If $A=A(t)=\sqrt{\frac{\Rb}{4t}}$, $\gamma $ is positive, $f_0(x,\gamma)$ is uniformly continuous in $\gamma$ and satisfies (\ref{L7.25}), then
\[
\Int_{\gamma'=0}^{\gamma}\Int_{\gamma''=0}^{+\infty} f_{0,x}(x,\gamma'')\left(e^{-A^2(\gamma''-\gamma')^2}-e^{-A^2(\gamma''+\gamma')^2}\right) {\rm d}\gamma''  {\rm d}\gamma'
\]
tends to 
\begin{equation}
\label{L20}
\Int_{\gamma''=0}^{+\infty} f_{0,x}(x,\gamma'') \Int_{\gamma'''=-\gamma''}^{\gamma''} e^{-A^2 \gamma'''^2}{\rm d}\gamma'''{\rm d}\gamma'',
\end{equation}
when $\gamma \rightarrow +\infty$.
\end{lemma}

The proof relies on Fubini's theorem and changes of variables for the
two integrals.  The proof is only technical and left to the
reader.

After simplification by $\beta$, the continuity of the vertical
velocity (\ref{l.17}) reads after making $\gamma \rightarrow + \infty$
thanks to Lemma \ref{b.2}:
\begin{multline}
\label{l.19}
\eta_t+\int_0^1 u^u_x+ \alpha (u^u(1) \eta)_x -\Frac{\varepsilon}{\sqrt{\pi R \sqrt{b}}} u^u_x(x,0,t) \ast \frac{1}{\sqrt{t}}\\
+\Frac{2\varepsilon}{\sqrt{\pi}}\Int_{\gamma''=0}^{+\infty} f_{0,x}(x,\gamma'') \Int_{\gamma'=0}^{A(t)\gamma''} e^{ -\gamma'^2}{\rm d}\gamma'{\rm d}\gamma''=O(\varepsilon^2),
\end{multline}
where the convolution is in time $t$ and the formula
$\mathcal{L}^{-1}\left(\frac{1}{\sqrt{p}}\right) =1/\sqrt{\pi t}$
\cite{Doetsch} is used. If one had made the more rigorous expansion
according to Remark \ref{rem_hypotheses}, assuming $u^u$ is defined
only on $(\varepsilon,1+\alpha \eta)$ and $u^b$ is defined on $(0,1)$,
one would have been led to
\begin{multline}
u_t+\int_{\varepsilon}^{1+\alpha \eta} u^u_x +\alpha u^u(1+\alpha \eta) \eta_x-\Frac{\varepsilon}{\sqrt{\pi R \sqrt{b}}} u^u_x(x,z=\varepsilon,t) \ast \frac{1}{\sqrt{t}}\\
+\Frac{2\varepsilon}{\sqrt{\pi}}\Int_{\gamma''=0}^{1/\sqrt{\varepsilon}} f_{0,x}(x,\gamma'') \Int_{\gamma'=0}^{A(t)\gamma''} e^{ -\gamma'^2}{\rm d}\gamma'{\rm d}\gamma''=O(\varepsilon^2).
\end{multline}

\subsection{The dependence on $z$ of the fields}
\label{subsec3.4}

At this stage, we have reduced the equations but not as much as in the
Euler case which leads to a Boussinesq system in 1+1 dimension.  We
have derived only a 2+1 dimension problem although we have
eliminated $w$ and $p$. The major difference with the Boussinesq
derivation comes from the assumption of irrotationnality of Euler
flows. This assumption would provide $u_z=O(\varepsilon)$. Such a
condition would annihilate the dependence on $z$ and greatly simplify
the above computations.

Yet irrotationality and its corollary of a potential flow is
incompatible with the number of conditions we set at the bottom, which
are needed by the dissipativity of the Navier-Stokes equations. So we
need to determine the dependence on $z$ of $u$ to have a more
tractable system.

Starting from now, we drop the $u$ superscripts for the fields in the
upper part but keep the superscripts for the boundary layer. In
summary, we assume $\RE \simeq \varepsilon^{-5/2}$, and the assumptions
of the first asymptotic stated in the subsection
\ref{subsect.a.1}. Up to now, the reduced equations are collected from
(\ref{u.4}) and (\ref{l.19}):
\begin{align}
\label{l.23}  & u_t+\eta_x +\alpha u u_x+\alpha u_z(\eta_t +\int_z^1u_x)-\beta \eta_{xtt}(z-1) -\beta \eta_{xxx}(z-1)^2 /2 = O(\varepsilon^2), \; \forall z \\
\nonumber
&\eta_t+\int_0^1 u_x(z)\, {\rm d} z+ \alpha (u(z=1) \eta)_x
-\Frac{\varepsilon}{\sqrt{\pi R\sqrt{b}}} u_x(x,z=0,t) \ast \frac{1}{\sqrt{t}} \\
\label{l.24} &\hspace*{2cm} +\Frac{2\varepsilon}{\sqrt{\pi}}\Int_{\gamma''=0}^{+\infty} \left(u^{b,0}_{x}(x,\gamma'') -u^{u,0}_{x}(x,z=0) \right) \Int_{\gamma'=0}^{A(t)\gamma''} e^{ -\gamma'^2}{\rm d}\gamma'{\rm d}\gamma''  =  O(\varepsilon^2).
%\\
%\label{l.25}
%u_z(1+\alpha \eta)+\beta \eta_{xt} & = & O(\varepsilon^2).
\end{align}
The equation (\ref{l.23}) can be rewritten thanks to the order 0 of
(\ref{l.24}):
\begin{equation}
\label{l.26}
u_t+\eta_x +\alpha u u_x-\alpha u_z \int_0^z u_x-\beta \eta_{xtt}(z-1)
-\beta \eta_{xxx}(z-1)^2 /2 =  O(\varepsilon^2), \forall z.
\end{equation}
Notice that the $\eta_{xxx}$ term comes from an integral of the shape
$\int_1^z \int_1^{z'} u_{xxt}$. As an intermediate result one may see
very easily that $\eta_{xx}=\eta_{tt}+O(\varepsilon)$ which is useful
later.

We intend to prove the following Lemma:

\begin{lemma}
A localized solution of (\ref{l.24}), (\ref{l.26}) is such that
\begin{align}
\label{l.28}
 \int_0^1 u = & u(x,z,t)-\beta \eta_{xt}\Frac{z^2-1/3}{2}+O(\varepsilon^2),\\
\label{l.29}
u(x,0,t) = & u(x,z,t) - \beta \eta_{xt} \frac{z^2}{2}+O(\varepsilon^2),\\
\label{l.30}
u(x,1,t) = & u(x,z,t)+\beta \eta_{xt}
\frac{1-z^2}{2}+O(\varepsilon^2).
\end{align}
\end{lemma}

\begin{proof}

In a preliminary step, we prove
\begin{equation}
\label{l.27}
u_z(x,z,t)=\beta \, \eta_{xt}(x,t) \, z+O(\varepsilon^2).
\end{equation}
To that end, we differentiate (\ref{l.26}) with respect to $z$, so as to have:
\[
u_{zt}+\alpha u\, u_{xz}-\alpha u_{zz}\, \int_0^z u_x -\beta\eta_{xtt}-\beta\eta_{xxx}(z-1)=O(\varepsilon^2),
\]
and we can integrate this equation in time using that
$\eta_{xx}=\eta_{tt}+O(\varepsilon)$:
\begin{equation}
\label{l.27.5}
u_z+\alpha\int_{t_0}^t (u \, u_{xz})-\alpha \int_{t_0}^t (u_{zz}\int_0^z u_x)-\beta\eta_{xt}-\beta \eta_{xt}(z-1)=C_3(x,z)+O(\varepsilon^2),
\end{equation}
where $C_3$ is a function of $x,z$ but it does not depend on $t$.
Since the solution is localized for any $x,z$, there exists a time
$t_0$ at which $u_z$ and $\eta_{xt}$ vanish or are as small as wanted
(in a local norm). So
\[
C_3(x,z)=O(\varepsilon),
\]
in a first attempt to determine $C_3$. But then the equation (\ref{l.27.5})
implies $u_z=O(\varepsilon)$ and so the quadratic terms are all
of second order in (\ref{l.27.5}) since they contain at least one $u_z$. Hence
\[
u_z(x,z,t)-\beta \eta_{xt} \, z = C_4(x,z)+O(\varepsilon^2).
\]
Again since for all $(x,z)$ there exists a time at which $u$ and
$\eta$ vanish or are as small as wanted, then
$C_4(x,z)=O(\varepsilon^2)$ and this completes the proof of
(\ref{l.27}).
We can then go further by integrating between $z'$ and $z$:
\[
u(x,z,t)=u(x,z',t)+\beta \eta_{xt}
\frac{z^2-z'^2}{2}+O(\varepsilon^2),
\]
and then, integrating in $z'$ between $z'=0$ and $z'=1$, we can state
(\ref{l.28}). Setting $z'=0$ (or $z'=\varepsilon$), we obtain
(\ref{l.29}) and setting $z'=1$ we obtain (\ref{l.30}).

\end{proof}

So the system (\ref{l.24},\ref{l.26}) can be rewritten thanks to
(\ref{l.28}-\ref{l.30}), the formula
$\mathcal{L}^{-1}\left(\frac{1}{\sqrt{p}}\right) =1/\sqrt{\pi t}$
\cite{Doetsch}, and the fact that, as in the Euler case the wave
equation is the zeroth order ($\eta_{xx}=\eta_{tt}+O(\varepsilon)$):
\begin{align}
\label{l.31}
&u_t+\eta_x +\alpha u u_x-\beta \eta_{xxx}\Frac{(z^2-1)}{2} =  O(\varepsilon^2),\\
\nonumber & \eta_t+u_x(x,z,t)-\Frac{\beta}{2} \eta_{xxt} (z^2-\frac{1}{3})+ \alpha (u \eta)_x -\Frac{\varepsilon}{\sqrt{\pi R\sqrt{b}}} u_x \ast \frac{1}{\sqrt{t}} \\  
& \label{l.32} +\Frac{2\varepsilon}{\sqrt{\pi}}\Int_{\gamma''=0}^{+\infty} \left(u^{b,0}_{x}(x,\gamma'') -u^{u,0}_{x}(x,z=0) \right) \Int_{\gamma'=0}^{A(t)\gamma''} e^{ -\gamma'^2}{\rm d}\gamma'{\rm d}\gamma'' =  O(\varepsilon^2),
\end{align}
where all the fields $u$ are evaluated at $(x,z,t)$ and the
convolution is in time. This is the system stated in Proposition
\ref{Prop.1} and the proof is complete.

\section{Generalization and checkings}

In a first subsection, we state the 2-D Boussinesq system and check we
may find the classical Boussinesq systems in the inviscid
case. Then, in Subsection \ref{subsec3.5} we derive rigorously the
viscous KdV equation and discuss its compatibility with the equation derived by
Kakutani and Matsuuchi in \cite{KM_75}, by Liu and Orfila in
\cite{Liu_Orfila_04}, and by Dutykh in \cite{Dutykh_09_a}.

\subsection{The full 2-D Boussinesq systems family}

One may start from the 3-D Navier-Stokes equations and derive in a way
very similar to above a generalization of (\ref{l.31}, \ref{l.32}):
\begin{equation}
\label{l.33}
\left\{
\begin{array}{rl}
u_t+\eta_x +\alpha u u_x+\alpha v u_y-\beta (\eta_{xxx}+\eta_{xyy})\Frac{(z^2-1)}{2} & =  O(\varepsilon^2),\\
v_t+\eta_y +\alpha u v_x+\alpha v v_y-\beta (\eta_{yxx}+\eta_{yyy})\Frac{(z^2-1)}{2} & =  O(\varepsilon^2),\\
\eta_t+u_x+v_y-\Frac{\beta}{2} (\eta_{xxt}+\eta_{yyt}) (z^2-\frac{1}{3})&\\
+ \alpha (u \eta)_x+ \alpha (v \eta)_y + \frac{\varepsilon}{\sqrt{\pi R\sqrt{b}}} \eta_t \ast \left(\frac{1}{\sqrt{t}}\right)&\\
+\Frac{2\varepsilon}{\sqrt{\pi}}\Int_{\gamma''=0}^{+\infty} \left(u^{b,0}_{x}(x,\gamma'') -u^{u,0}_{x}(x,z=0) \right) \Int_{\gamma'=0}^{A(t)\gamma''} e^{ -\gamma'^2}{\rm d}\gamma'{\rm d}\gamma'' & =  O(\varepsilon^2).
\end{array}
\right.
\end{equation}

In case of a Euler initial condition, the last integral term vanishes,
but this is not physical as is stressed in Remark
\ref{remark_Euler_NS}.

It is well-known thanks to \cite{BCS02} that there is a family of
Boussinesq systems, indexed by three free parameters. All these
systems are equivalent in the sense that up ot order 1, they can be
derived one from the other by using their own $O(\varepsilon^0)$ part
and by replacing partially $\eta_t, \eta_x$ and $\eta_y$ by $u_x, u_t,
v_t$. We are going to prove the same for our system. Namely, the order
$0$ of (\ref{l.33}) enables to interpolate with $a_{int}, b_{int},
c_{int}$:
\[
\left\{
\begin{array}{l}
\eta_x =a_{int} \eta_x-(1-a_{int})u_t+O(\varepsilon),\\
\eta_y =b_{int}\eta_y-(1-b_{int})v_t+O(\varepsilon),\\
\eta_t=c_{int}\eta_t-(1-c_{int})(u_x+v_y)+O(\varepsilon).
\end{array}
\right.
\]
These formulas are reported in the full 2D system (\ref{l.33}), where
we drop the convolution term and the integral on the initial
condition:
\begin{equation}
\label{l.34}
\left\{
\begin{array}{ll}
u_t+\eta_x +\alpha u u_x+\alpha v u_y-a_{int} \beta  \Delta \eta_{x}\Frac{(z^2-1)}{2} &\\
+(1-a_{int})\beta \Delta u_{t}\Frac{(z^2-1)}{2} & =  O(\varepsilon^2),\\
v_t+\eta_y +\alpha u v_x+\alpha v v_y-b_{int} \beta \Delta \eta_{y}\Frac{(z^2-1)}{2} & \\
+(1-b_{int}) \beta \Delta v_t\Frac{(z^2-1)}{2} & =  O(\varepsilon^2),\\
\eta_t+u_x+v_y-c_{int} \Frac{\beta}{2} \Delta \eta_{t}(z^2-\frac{1}{3})& \\
+ (1-c_{int}) \Frac{\beta}{2} \Delta (u_x+v_y)(z^2-\frac{1}{3})+ \alpha (u \eta)_x+ \alpha (v \eta)_y  & = O(\varepsilon^2),
\end{array}
\right.
\end{equation}
where we denote $\Delta$ the $x,y$ Laplacian.

This is the general Boussinesq system as can be seen in \cite{BCS02}
(p. 285 equation (1.6)). Indeed if we denote $a_{BCS}, b_{BCS},
c_{BCS}$ and $d_{BCS}$ the interpolation parameters of this article,
we can identify the 1D version of our interpolated (\ref{l.34}) with
\[
\begin{array}{ll}
a_{BCS}= \frac{\beta}{2} (1-c_{int})(z^2-\frac{1}{3}) ,& b_{BCS}= \frac{\beta}{2} c_{int}(z^2-\frac{1}{3}),\\
c_{BCS}= -\beta a_{int}\frac{z^2-1}{2}  ,& d_{BCS}= -\beta (1-a_{int})\frac{z^2-1}{2}.
\end{array}
\]
The meaning of our height $z$ is the same as the $\theta$ of \cite{BCS02} and the relation between $a_{BCS}, b_{BCS}, c_{BCS}$ and $d_{BCS}$ (see (1.8) of this article) is satisfied.

\subsection{About the KdV-like equation}
\label{subsec3.5}

Various authors have derived either a viscous Boussinesq system or a
viscous KdV equation. 

One may wonder what is the viscous KdV equation derived from our
viscous Boussinesq system and compare it with what may be found in the
literature. First, we state and prove the following Proposition.

\begin{proposition}
\label{prop_KdV}
If the initial flow is localized, the KdV change of variables applied
to the system (\ref{l.31}, \ref{l.32}) leads to
\begin{equation}
\label{kdv.3}
2\tilde{\eta}_{\tau}+3a \tilde{\eta}
\tilde{\eta}_{\xi}+\frac{b}{3}\tilde{\eta}_{\xi \xi \xi}
-\frac{1}{\sqrt{\pi R \sqrt{b}}}\int_{\xi'=0}^{\tau/\varepsilon}\frac{\tilde{\eta}_{\xi}(\xi+\xi',\tau)}{\sqrt{\xi'}}{\rm d}\xi'=O(\varepsilon),
\end{equation}
for not too small times $\tau$, where we set $\alpha=a \varepsilon$,
Re$=R\,  \varepsilon^{-5/2}$ and $\beta=b\varepsilon$.
\end{proposition}

In the formula (\ref{kdv.3}), since it has been proved in
\cite{Lannes_2013} that KdV is a good approximation of Euler for times
up to $1/\varepsilon^2$, and that the velocity is localized, it is a
strong temptation to replace the integral term by
\[
-\frac{1}{\sqrt{\pi R \sqrt{b}}} \int_{\xi'=0}^{+\infty} \frac{\tilde{\eta}_{\xi}(\xi+\xi',\tau)}{\sqrt{\xi'}}{\rm d}\xi'.
\]

This is the term found in \cite{KM_75}.

\begin{proof}

We start from the most general form of (\ref{l.31}, \ref{l.32})
(same as (\ref{Boussinesq_1})) and use the KdV change of variables
\begin{equation}
\label{kdv.1}
(\xi=x-t, \; \tau=\varepsilon t) \; \Leftrightarrow
(x=\xi+\tau/\varepsilon, \; t=\tau/\varepsilon),
\end{equation}
and change of fields
\begin{equation}
\label{kdv.2}
\Phi(x,z,t)=\tilde{\Phi}(x-t,z,\varepsilon t) \Rightarrow
\Phi_t=-\tilde{\Phi}_{\xi}+\varepsilon \tilde{\Phi}_{\tau}(x-t,z,\varepsilon
t),
\end{equation}
where the generic field $\Phi$ is tilded when it depends on the
$(\xi,z,\tau)$ variables.

There are only two difficult terms in the system (\ref{l.31},
\ref{l.32}) (equivalent to (\ref{Boussinesq_1})). The first is the
convolution which we denote $T_1$:
\[
\begin{array}{rcl}
T_1(x,z,t) & = & -\frac{\varepsilon}{\sqrt{\pi R\sqrt{b}}} \Int_{t'=0}^{t} \Frac{u_x(x,z,t-t')}{\sqrt{t'}}  {\rm d}t'\\
           & = & -\frac{\varepsilon}{\sqrt{\pi R\sqrt{b}}} \Int_{t'=0}^{t} \Frac{\tilde{u}_{\xi}(x-t+t',z,\varepsilon t-\varepsilon t')}{\sqrt{t'}}  {\rm d}t',
\end{array}
\]
because of (\ref{kdv.1}). But then it suffices to recognize the function
of $(x-t,\varepsilon t)=(\xi ,\tau)$ in the last equation to
have the term after the KdV change of variables:
\begin{equation}
\label{kdv.4}
\begin{array}{rcl}
\tilde{T}_1(\xi,z,\tau) & = & -\frac{\varepsilon}{\sqrt{\pi R\sqrt{b}}} \Int_{t'=0}^{\tau/\varepsilon}
  \Frac{\tilde{u}_{\xi}(\xi+t',z,\tau-\varepsilon t')}{\sqrt{t'}} {\rm d}t'\\
 & = & -\frac{\varepsilon}{\sqrt{\pi R\sqrt{b}}} \Int_{\xi'=0}^{\tau/\varepsilon} \Frac{\tilde{u}_{\xi}(\xi+\xi',z,\tau)}{\sqrt{\xi'}} {\rm d}\xi'+O(\varepsilon^2).
\end{array}
\end{equation}
Since the $t'$ variable is in place of a $\xi$, we changed the
notation to $\xi'$. This term is odd because it has an integration
variable ($\xi'$) that has a physical meaning but bounds depending on
time $\tau/\varepsilon$. We discuss it below.

The second difficult term is the one that keeps the initial conditions
and writes:
\[
\begin{array}{rl}
T_2(x,z,t)        =&+\Frac{2\varepsilon}{\sqrt{\pi}}\Int_{\gamma''=0}^{+\infty} \left(u^{b,0}_{x}(x,\gamma'') -u^{u,0}_{x}(x,z=0) \right) \times \Int_{\gamma'=0}^{\sqrt{\frac{R\sqrt{b}}{4t}} \gamma''} e^{ -\gamma'^2}{\rm d}\gamma'{\rm d}\gamma''.
\end{array}
\]
The change of variables (\ref{kdv.1}) gives:
\[
\begin{array}{rl}
\tilde{T}_2(\xi,z,\tau) =&+\Frac{2\varepsilon}{\sqrt{\pi}}\Int_{\gamma''=0}^{+\infty} \left(u^{b,0}_{x}(\xi+\frac{\tau}{\varepsilon},\gamma'') -u^{u,0}_{x}(\xi+\frac{\tau}{\varepsilon},z=0) \right) \\
 & \times \Int_{\gamma'=0}^{\sqrt{\frac{R\sqrt{b}\, \varepsilon}{4\tau}} \gamma''} e^{ -\gamma'^2}{\rm d}\gamma'{\rm d}\gamma''.
\end{array}
\]
If the initial boundary layer is localized, for $\tau$ not too small,
$ u^{b,0}_{x}(\xi+\frac{\tau}{\varepsilon},\gamma'')
-u^{u,0}_{x}(\xi+\frac{\tau}{\varepsilon},z=0) $ will be small in
$L^1_{\gamma''}$ and so $\tilde{T}_2$ will be negligible in comparison
with $\varepsilon$ and so can be dropped. In addition, the inner
integral's upper bound is very close to the lower bound.

Then, we can claim that the Boussinessq system after the KdV change of
variables and fields is
\begin{equation}
\label{kdv.5}
\left\{
\begin{array}{rl}
-\tilde{u}_{\xi}+\varepsilon \tilde{u}_{\tau}+\tilde{\eta}_{\xi}+\alpha \tilde{u} \tilde{u}_{\xi} -\beta \tilde{\eta}_{\xi \xi
 \xi}\left(\Frac{z^2-1}{2}\right)&=O(\varepsilon^2), \\
-\tilde{\eta}_{\xi}+\varepsilon \tilde{\eta}_{\tau}+\tilde{u}_{\xi}+\Frac{\beta}{2}\tilde{\eta}_{\xi \xi
 \xi}\left(z^2-\frac{1}{3}\right)+\alpha(\tilde{u} \tilde{\eta})_{\xi}&\\
\hspace*{1.5cm}-\frac{\varepsilon}{\sqrt{\pi R\sqrt{b}}} \Int_{\xi'=0}^{\tau/\varepsilon} \Frac{\tilde{u}_{\xi}(\xi+\xi',z,\tau)}{\sqrt{\xi'}} {\rm d}\xi'&=O(\varepsilon^2).
\end{array}
\right.
\end{equation}
We may notice that at the first order, and as in the derivation
of the KdV equation,
\[
\tilde{u}_{\xi}=\tilde{\eta}_{\xi}+O(\varepsilon) \Rightarrow
\tilde{u}=\tilde{\eta} +O(\varepsilon),
\]
thanks to a simple and classical integration (and a localized
solution). But then the sum of the two equations of (\ref{kdv.5})
gives:
%
%\[
%\begin{array}{l}
\begin{multline*}
\varepsilon \tilde{u}_{\tau} + \varepsilon \tilde{\eta}_{\tau}+\alpha \tilde{u} \tilde{u}_{\xi}+\alpha (\tilde{u} \tilde{\eta })_{\xi}+\Frac{\beta}{3} \tilde{\eta}_{\xi \xi \xi}-
\frac{\varepsilon}{\sqrt{\pi R\sqrt{b}}} \Int_{\xi'=0}^{\tau/\varepsilon} \Frac{\tilde{u}_{\xi}(\xi+\xi',z,\tau)}{\sqrt{\xi'}} {\rm d}\xi'=O(\varepsilon^2).
\end{multline*}
%\end{array}
%\]
%
Using now the fact that $\tilde{u}=\tilde{\eta}+O(\varepsilon)$,
dividing by $\varepsilon$, one states exactly the equation
(\ref{kdv.3}). The convolution that used to be on time is now
on $\xi'$ and the proof is complete.
\end{proof}

What can be found in the literature ?

As stated in the introduction, various authors already derived either
a viscous Boussinesq system or a viscous KdV equation. Yet, none of
them have the very same equation as us. We must clarify why there are
such differences.

The first article is \cite{Ott_Sudan_70} in which Ott and Sudan obtained formally in nine lines:
\[
+\alpha_3 \Int_{\xi'=-\infty}^{+\infty} \Frac{\tilde{u}_{\xi}(\xi',\tau)\sgn (\xi-\xi')}{\sqrt{\mid \xi-\xi' \mid}} {\rm d}\xi'.
\]
but the authors used a Fourier transform \cite{KM_75} (p. 243) and
they made an error pointed by \cite{KM_75}. Our formula differs from
Ott and Sudan's by the sign and the bound !

Later, Kakutani and Matsuuchi \cite{KM_75} derived rather rigorously
the KdV equation from Navier-Stokes and we set the same regime as
them. Yet, they did not raise the problem of the initial condition. As
a consequence, they used a time-Fourier transform to solve the
heat-like equation. They proposed:
\[
\Frac{1}{4\sqrt{\pi R}}\Int_{\xi'=-\infty}^{+\infty}\Frac{\tilde{\eta}_{\xi}(\xi',\tau)(1-\sgn (\xi-\xi'))}{\sqrt{\mid \xi-\xi' \mid}} {\rm d}\xi'.
\]
Their half order derivative term differs from ours only by
the bound of the integral which is $\tau/ \varepsilon$ for us and
$+\infty$ for them. 

Liu and Orfila in \cite{Liu_Orfila_04} (and subsequent articles)
derived a Boussinesq system for a regime different from ours (Re=$R\,
\varepsilon^{-7/2}$). They solved their heat equation with a
sine-transform in the vertical coordinate by quoting \cite{Mei_95}
where a vanishing initial condition is assumed. Given their regime,
their Boussinesq system is right. But when they derived a KdV equation
(see \cite{Liu_Orfila_04} p. 89), they did not make explicit their
change of variables in the term equivalent to our $T_1$. With the
change of variable $\xi_{LO}=x-t, \tau_{LO}=(\alpha_{LO}/\mu_{LO}) t$,
they exhibit (see their (3.19) or (3.21)):

\[
-\Frac{1}{2\sqrt{\pi}} \Int_0^t \Frac{\eta_{\xi_{LO}}}{\sqrt{t-T}} {\rm d}T,
\]
where there remains the former variable $t$ inside the integral and in
the bounds. Moreover, the dependence of $\eta_{\xi_{LO}}$ on the
variables ($t, \tau_{LO}, ...$ ?) is not written. Is the $T$ variable
in the integral a time variable ? One may wonder whether they did
notice that the time convolution transforms into a {\em space} one.

Dutykh derived a Boussinesq system by a Leray-Helmholtz decomposition
from a Linearized Navier-Stokes \cite{Dutykh_09_a}. In order to derive
the associated KdV (see Sec. 3.2), he assumed $u=\eta+\varepsilon
P+\beta Q+...$ and found $P$ and $Q$. In this process, he used only the
assumption that waves go right ($\eta_t+\eta_x=O(\varepsilon)$). So he
{\em did not} use the change of time ($\tau=\varepsilon t$) and wrote a
formula with unscaled time $t$ (his (14)):
\[
-\sqrt{\Frac{\nu}{\pi} \, \Frac{g}{h}}\Int_0^t \Frac{\eta_x}{\sqrt{t-\tau}} {\rm d} \tau.
\]
Similar criticisms can be said on this formula in which the integral
seems to be on time while it should be on the shifted space $\xi$.

\section{Conclusion}
In this article, we derive the viscous Boussinesq system for surface
waves from Navier-Stokes equations with non-vanishing initial
conditions (see Proposition \ref{Prop.1}). One of our by-product is
the bottom shear stress as a function of the velocity (cf. Proposition
\ref{Prop.2}) and the decay rate for shallow water (see Proposition
\ref{prop1}). We also state the system in 3-D in (\ref{l.33}),
and derive the viscous KdV equation from our viscous Boussinesq system
(cf. Proposition \ref{prop_KdV}). The differences of our viscous KdV
with other equations, already derived in the literature, are highlighted
and explained.

\appendix

\section{Boundary and initial conditions in Lemma \ref{b.1}}
\label{appendix_A}

As is said in the proof of Lemma \ref{b.1}, we must check that $u$,
given by the necessary equation (\ref{L18}), satisfies the initial
condition (\ref{L10}) and the remaining of the boundary conditions
(\ref{L11})$_2$.

\underline{Concerning the initial condition (\ref{L10})}. We try to
find the limit when $t$ tends to $0^+$ and then $A=A(t)$ tends to
$+\infty$. Since one assumes below $\gamma>0$,
\[
-u^u(x,0,.) \ast \mathcal{L}^{-1}(e^{-\sigma \gamma})=-u^u(x,0,t) \ast\frac{\sqrt{R\sqrt{b}}}{2\sqrt{\pi} t^{3/2}}e^{-\frac{R\sqrt{b}}{4t}}
\]
tends to zero exponentially (the convolution is the Laplace one and on
time $t$). Then, one can come back to the formula of $f$
(\ref{L17}) and make one change of variables in every integral:
\begin{align*}
f(x,\gamma,t)=&\Frac{A}{\sqrt{\pi}} \Int^{+\infty}_{-\gamma} f_{0}(x,\Gamma'+\gamma)e^{-A^2 \Gamma'^2} {\rm d} \Gamma' \\
&-\Frac{A}{\sqrt{\pi}} \Int_{\gamma}^{+\infty} f_{0}(x,\Gamma'-\gamma)e^{-A^2 \Gamma'^2} {\rm d} \Gamma'+O(\varepsilon),
\end{align*}
up to an exponentially tending to zero function when $t$ tends to $0$ thanks to
$A(t)$. This can be rewritten
\begin{align*}
f(x,\gamma,t)=& \Frac{A}{\sqrt{\pi}} \Int^{+\infty}_{\gamma} \left(f_{0}(x,\Gamma'+\gamma)-f_{0}(x,\Gamma'-\gamma)\right) e^{-A^2 \Gamma'^2} {\rm d} \Gamma' \\
& +\Frac{A}{\sqrt{\pi}} \Int_{-\gamma}^{\gamma} f_{0}(x,\Gamma'+\gamma)e^{-A^2 \Gamma'^2} {\rm d} \Gamma'+O(\varepsilon),
\end{align*}
where we denote $I_2$ the second integral. The first integral may be
bounded by
\begin{align*}
     & \Frac{2A}{\sqrt{\pi}} \Sup_{\gamma > 0} \mid f_{0}(x,\gamma) \mid \Int_{\gamma}^{+\infty} e^{-A^2\Gamma'^2}{\rm d} \Gamma'\\
\leq & \Frac{2}{\pi} \Sup_{\gamma > 0} \mid f_{0}(x,\gamma) \mid \Int_{A\gamma}^{+\infty} e^{-\Gamma''^2}{\rm d} \Gamma'',
\end{align*}
which clearly tends to zero when $t$ tends to zero
($A=A(t)\rightarrow +\infty$).\\
For the second integral denoted $I_2$, one may compute a similar
integral where the integration variable of $f_0$ is frozen:
\[
\begin{array}{rl}
I_2'= & \Frac{A}{\sqrt{\pi}} \Int_{-\gamma}^{\gamma} f_{0}(x,\gamma)e^{-A^2 \Gamma'^2} {\rm d} \Gamma'\\
=&f_{0}(x,\gamma)\Frac{1}{\sqrt{\pi}}\Int_{-A\gamma}^{A\gamma}e^{- \Gamma''^2} {\rm d} \Gamma'',
\end{array}
\]
which clearly tends to $f_0(x,\gamma)$ if $\gamma>0$ when
$t\rightarrow 0^+$. So one may make the difference of the second
integral $I_2$ with the previous integral (which tends to
$f_0(x,\gamma)$) and find:
\[
I_2-I_2'= \Frac{A}{\sqrt{\pi}} \Int_{-\gamma}^{\gamma} \left(f_{0}(x,\Gamma'+\gamma)-f_{0}(x,\gamma)\right)e^{-A^2 \Gamma'^2} {\rm d} \Gamma' + o_{t \rightarrow 0^+}(1).
\]
Here we must use the assumption of uniform continuity of the initial data:
\[
\forall \epsilon>0 \; \exists \delta>0 \; / \; \mid \gamma'-\gamma \mid< \delta \Rightarrow \mid f_{0}(x,\gamma')-f_{0}(x,\gamma) \mid < \epsilon.
\]
Then, for any $\epsilon >0$, there exists a $\delta$ such that $I_2-I_2'$ can
be splitted into two parts and bounded by
\[
\begin{array}{rll}
 & \Frac{2A}{\sqrt{\pi}} \Sup_{\gamma > 0} \mid f_{0}(x,\gamma) \mid \int_{\mid \Gamma' \mid > \delta \bigcap \mid \Gamma' \mid <\gamma } e^{-A^2 \Gamma'^2} {\rm d} \Gamma' &+ \Frac{A}{\sqrt{\pi}} \epsilon \Int_{ \mid \Gamma' \mid <\delta \bigcap  \mid \Gamma' \mid <\gamma } e^{-A^2 \Gamma'^2} {\rm d} \Gamma'\\[3mm]
\leq & \Frac{2A}{\sqrt{\pi}}\Sup_{\gamma > 0} \mid f_{0}(x,\gamma) \mid 2\gamma e^{-A^2 \delta^2} &+ \; \epsilon.
\end{array}
\]
So we have proved that the $f$ given by (\ref{L17}) or $u$ given by
(\ref{L18}) satisfies the initial condition ($A(t) \rightarrow +
\infty$ when $t \rightarrow 0^+$).

\underline{Concerning the boundary condition (\ref{L11})$_2$}). Now we
look for the limit when $\gamma$ tends to $+\infty$. The formula
(\ref{L17}) can be written:
\[
\hat{f}(x,\gamma, p)= \Frac{\Rb}{2 \sigma} \Int^{+\infty}_{\gamma} f_{0}(x,\gamma')e^{-\sigma \gamma'} {\rm d} \gamma' e^{+\sigma \gamma}+ \Frac{\Rb}{2 \sigma} \Int_{0}^{\gamma} f_{0}(x,\gamma')e^{\sigma \gamma'} {\rm d} \gamma' e^{-\sigma \gamma},
\]
up to some exponentially tending to zero functions of $\gamma$. In this
formula, the first integral is bounded by
\[
\begin{array}{rl}
     & \Frac{\Rb}{2\sigma} \Sup_{\gamma' \geq \gamma} \mid f_{0}(x,\gamma') \mid \int_{\gamma}^{+\infty} e^{-\sigma \gamma'} {\rm d} \gamma' e^{\sigma \gamma}\\
\leq &  \Frac{\Rb}{2\sigma^2} \Sup_{\gamma' \geq \gamma} \mid f_{0}(x,\gamma') \mid,
\end{array}
\]
which clearly tends to zero when $\gamma$ tends to $+\infty$ because
$f_0(x,\gamma)$ tends to zero when $\gamma$ tends to $+\infty$.\\
For the second integral, one needs to cut it at a value
$\Gamma$ given by the definition of $f_0 \rightarrow 0$ when $\gamma$
tends to $+\infty$ ($\forall \epsilon >0 \; \exists \Gamma >0 \; / \mid
\gamma \mid > \Gamma \Rightarrow \mid f_0 \mid < \epsilon$). We can
bound it with:
\[
\Frac{\Rb}{2\sigma} \Int_0^{\Gamma} \mid f_0(x,\gamma') \mid e^{\sigma \gamma'} {\rm d} \gamma'e^{-\sigma \gamma}+\Frac{\Rb}{2\sigma} \epsilon \Int_{\Gamma}^{\gamma} e^{\sigma \gamma'} {\rm d} \gamma'  e^{-\sigma \gamma}.
\]
Since the first term tends to zero when $\gamma$ tends to $+\infty$
($\Gamma$ fixed) and the second term is less than $\Rb \epsilon
/(2\sigma^2)$, the whole can be made smaller than any $\epsilon$.

So the proof that (\ref{L11})$_2$ is satisfied is complete.

\section*{Acknowledgement}

The author wants to thank Professor Jean-Claude Saut for initiating
and following this research and Professor David G\'erard-Varet for a
fruitful discussion.

%\newpage

%\begin{figure}[htbp]
%\begin{center}
%%\includegraphics[width=7cm]{dessins/fig1.eps}
%\includegraphics[width=7cm]{fig1.eps}
%\end{center}
%\caption{The dimensionless domain}
%\label{fig1}
%\end{figure}

%\newpage

%\bibliographystyle{plain}	% (uses file "plain.bst")
%\bibliography{art}		% expects file "myrefs.bib"

\end{document}